\documentclass[11pt]{amsart}
\usepackage[margin=1.15in]{geometry}
\usepackage{amsmath,amssymb,amsthm,mathtools}
\usepackage{enumitem}
\usepackage{xcolor}
\usepackage{tikz}
\usepackage[nocompress]{cite}
\usepackage[colorlinks=true,linkcolor=blue!50!black,citecolor=green!40!black,urlcolor=blue!50!black]{hyperref}
\usetikzlibrary{svg.path}

\newtheorem{theorem}{Theorem}[section]
\newtheorem{proposition}[theorem]{Proposition}
\newtheorem{lemma}[theorem]{Lemma}

\theoremstyle{definition}
\newtheorem{definition}[theorem]{Definition}
\theoremstyle{remark}
\newtheorem{remark}[theorem]{Remark}
\numberwithin{equation}{section}

\newcommand{\R}{\mathbb{R}}
\newcommand{\C}{\mathbb{C}}
\newcommand{\Ball}{\mathbb{B}^3}
\newcommand{\Sph}{\mathbb{S}^2}
\newcommand{\Sig}{\Sigma}
\newcommand{\Om}{\Omega}
\newcommand{\bdSig}{\partial\Sigma}
\newcommand{\dA}{\,dA}
\newcommand{\ds}{\,ds}
\newcommand{\Ind}{\operatorname{Ind}}
\newcommand{\ind}{\operatorname{ind}}
\newcommand{\Null}{\operatorname{Null}}
\newcommand{\tr}{\operatorname{tr}}
\newcommand{\Span}{\operatorname{span}}
\newcommand{\eps}{\varepsilon}
\newcommand{\cC}{\mathcal C}
\newcommand{\cT}{\mathcal T}
\newcommand{\cS}{\mathcal S}
\newcommand{\cR}{\mathcal R}
\newcommand{\cP}{\mathcal P}
\newcommand{\cV}{\mathcal V}
\newcommand{\cK}{\mathcal K}
\newcommand{\mob}{\mathfrak m}
\newcommand{\Dbh}{D_{b\mathrm h}}
\newcommand{\Rot}{\mathrm{Rot}_b}
\newcommand{\hRot}{\widehat{\mathrm{Rot}}_b}

\title[The Morse index of prismatic free boundary minimal surfaces]{The Morse index of prismatic free boundary\\ minimal surfaces in the unit ball}
\author{Hung Tran}
\address{Department of Mathematics and Statistics, Texas Tech University, Lubbock, TX 79409, USA}
\email{hung.tran@ttu.edu}
\date{October 2026}

\begin{document}

\begin{abstract}
For integer $b\ge3$, let $\Sig\subset\Ball$ be an embedded free boundary minimal surface (FBMS) of genus zero with $b$ boundary
components. Suppose that $\Sig$ is invariant under the prismatic group of order $4b$, the rotation of order $b$ permutes
the boundary components cyclically, and the reflection in the equatorial plane preserves each of them. We show that the
Morse index of $\Sig$ is $2b$ and its nullity is three; moreover, we determine the negative space of the index form as a
representation of the symmetry group. Examples include the $b$-noids of Karpukhin,
Kusner, McGrath, and Stern and, for large $b$, the genus-zero surfaces of Folha, Pacard, and Zolotareva. To our
knowledge, apart from the equatorial disc and the critical catenoid, these are the first FBMS
in the ball whose Morse indices are known precisely. There are also lower bounds for the index of surfaces with
rotational or polyhedral symmetry. The main new ingredient is an equivariant comparison, using the representation theory of finite groups, between area and energy, in which the Teichm\"uller directions
are counted by symmetry type using circle domains. 
\end{abstract}

\maketitle

\section{Introduction}\label{sec:intro}

A \emph{free boundary minimal surface} (FBMS) in the closed unit ball $\Ball\subset\R^3$ is a compact minimal surface
$\Sig\subset\Ball$ with $\bdSig=\Sig\cap\Sph$ which meets the boundary sphere $\Sph=\partial\Ball$ perpendicularly. Such a
surface is a critical point of the area functional among surfaces in $\Ball$ whose boundaries are free to move on $\Sph$.
The subject has received tremendous interest since A.~Fraser and R.~Schoen related these surfaces to extremal metrics for
the first Steklov eigenvalue \cite{FS11,FS16}; see the surveys \cite{Li20,Fra20} and references therein. Throughout this
paper, all surfaces are assumed to be connected and embedded.

From the viewpoint of Morse theory, it is essential to study the second variation of area. For a unit normal $\nu$ and a
normal variation $f\nu$, it is given by the index form
\begin{equation}\label{eq:S}
  S(f,f)=\int_\Sig\bigl(|\nabla f|^2-|A|^2f^2\bigr)\dA-\int_{\bdSig}f^2\ds ,
\end{equation}
where $A$ denotes the second fundamental form. The \emph{Morse index} $\Ind(\Sig)$ is the number of negative eigenvalues
of $S$, counted with multiplicity; equivalently, it is the maximal dimension of a subspace on which $S$ is negative
definite. 
The \emph{nullity} $\Null(\Sig)$ is the dimension of the kernel $\ker S$.

The simplest examples are the equatorial \emph{flat disc}, of index one, and the \emph{critical catenoid}, the unique piece of a
catenoid centered at the origin which meets $\Sph$ orthogonally. 
Many further embedded examples have been constructed via
variational methods \cite{Ket16,CFS22,Chu23,FKS24,FS26}, gluing \cite{FPZ17,KL21,KM23,KW23,CSW25b,CSW25c,KZ26}, and
eigenvalue optimization \cite{FS16,KS24,KKMS24}. Remarkably, M.~Karpukhin, R.~Kusner, P.~McGrath, and D.~Stern showed that
every compact orientable surface with non-empty boundary is realized as an embedded FBMS in
$\Ball$ \cite{KKMS24}.

In contrast, much less is known about the Morse index. Fraser and Schoen proved that an FBMS other
than a flat disc has index at least four \cite{FS16}. The index of the critical catenoid is exactly four, as independently shown by Devyver \cite{Dev19}, G.~Smith and
D.~Zhou \cite{SZ19}, and the author \cite{Tra20}; moreover, its nullity is two \cite[Thm.~1.3]{Tra20}. We refer to
\cite{SSTZ21} for the higher-dimensional critical catenoids. There are numerous estimates on the index, developed in the last decade, including \cite{Sar17, ACS18, Lim22, Fra23, CSW25a, CSW25c, CSW25d}. Nevertheless, as noted in \cite[\S1]{CSW25a}, apart from the equatorial disc and the critical catenoid, no FBMS in $\Ball$ had a known Morse index.

In this paper, we compute the Morse index and nullity of a family of genus-zero surfaces with prismatic symmetry. 
For $b\ge3$, let $\Dbh\subset O(3)$ be the prismatic group of order $4b$, generated by the reflections in $b$ 
planes, which contain the $x_3$-axis $\ell$ and are equally spaced at angles $\pi/b$, and the reflection
$\sigma_h$ in the plane $\{x_3=0\}$. In particular, $\Dbh$ contains the rotation $\Rot$ by $2\pi/b$ about $\ell$. 

\begin{definition}\label{def:Sb}
For $b\ge3$, let $\cS_b$ be the set of embedded FBMS $\Sig\subset\Ball$ of genus zero with $b$
boundary components which are invariant under $\Dbh$, such that $\Rot$ permutes the boundary components cyclically and
$\sigma_h$ maps each boundary component to itself.
\end{definition}

\begin{remark}\label{rem:examples}
The class $\cS_b$ is non-empty for every $b\ge3$. Here are the relevant constructions.
\begin{itemize}[leftmargin=*,itemsep=2pt]
\item Karpukhin, Kusner, McGrath, and Stern maximized a normalized eigenvalue among the
metrics on a surface of genus zero with $b$ boundary components which are invariant under a given action of
$\mathbb Z_2\times D_b$. For the action of type $N_{\rho_1}(\mathbb Z_2\times D_b,\rho_1\rho_2)$ in their notation, the maximum is attained and is induced by an equivariant free boundary minimal embedding
\cite[Thm.~9.15(7)]{KKMS24}. The resulting surface, called the $b$-noid, belongs to $\cS_b$. 
\item For every sufficiently large $b$, the genus-zero surfaces constructed by A.~Folha, F.~Pacard, and T.~Zolotareva
\cite[Thm.~1.1]{FPZ17} belong to $\cS_b$. Each of them consists of two nearly parallel horizontal discs joined by $b$
small bridges placed periodically along the equator. It is invariant under the group generated by $\sigma_h$, a
half-turn about a horizontal axis, and the rotations preserving a regular $b$-gon; this group is $\Dbh$. 
Furthermore, $\sigma_h$ preserves each boundary component and $\Rot$ permutes them cyclically.
\item G.~Franz and M.~Schulz constructed by min--max methods, for every $b\ge3$, an embedded FBMS
of genus zero with $b$ boundary components aligned along the equator, which is invariant under the group of order $2b$
generated by $\Rot$ and a half-turn about a horizontal line \cite[Thm.~5.1]{FS26}. They conjecture that its full symmetry
group is $\Dbh$ \cite[Rem.~5.2]{FS26}.
\end{itemize}
Whether the constructions of \cite{FPZ17}, \cite{FS26}, and \cite{KKMS24} produce the same surfaces is an open question. 
\end{remark}


Our approach involves the representation theory of a finite group. Towards that end,  we recall some notation (see \S\ref{ssec:reps} for more details). Let $D_b\subset\Dbh$ be the dihedral group generated by the reflections in the $b$ vertical planes. 
We write $A_1$ and $A_2$ for the
trivial and sign representations of $D_b$, $E_1$ for its representation on the plane $\ell^\perp$, and $\cR_b$ for its regular
representation. For a representation $\pi$ of $D_b$, $\pi^\pm$ denotes the representation of
$\Dbh=D_b\times\langle\sigma_h\rangle$ which restricts to $\pi$ on $D_b$ and on which $\sigma_h$ acts by $\pm1$. Finally,
$\Dbh$ acts on functions on $\Sig$ by $g\cdot f=f\circ g^{-1}$, and the index form is invariant under this action.

\begin{theorem}\label{thm:main}
Let $b\ge3$ and $\Sig\in\cS_b$. Then
\[
  \Ind(\Sig)=2b .
\]
More precisely, the negative space of $S$ is isomorphic, as a representation of $\Dbh$, to
$(\cR_b\ominus A_2)^+\oplus A_1^-$. That is, its $\sigma_h$-even part has dimension $2b-1$ and contains each irreducible
representation $\pi\ne A_2$ of $D_b$ with multiplicity $\dim\pi$, while its $\sigma_h$-odd part is spanned by a single
$D_b$-invariant function. Furthermore, $\ker S$ is spanned by the functions $\langle a\times X,\nu\rangle$, $a\in\R^3$,
induced by the rotations; in particular, $\Null(\Sig)=3$.
\end{theorem}

\begin{remark}\label{rem:main}
To our knowledge, apart from the
equatorial disc and the critical catenoid, these are the first FBMS in $\Ball$ whose Morse
index and nullity are known. This also confirms part of \cite[Conjecture 7.8(iv)]{CSW25b}. 
\end{remark}

\begin{remark}
	The representation-theoretic framework of Theorem \ref{thm:main} remains valid for $b=2$. Here, $(\cR_2\ominus A_2)^+\oplus A_1^-\cong A_1^+\oplus B_1^+\oplus B_2^+\oplus A_1^-$; $A_1^+$ corresponds to the constant function while other components correspond to components of the normal vector. Indeed, it is possible to adapt the strategy in the proof of Theorem \ref{thm:main} for $b=2$ to give another verification that the Morse index of the critical catenoid is four. 
\end{remark}

\begin{remark}\label{rem:bounds}
Two general bounds bracket the index. The lower bound $\Ind(\Sig)\ge2b$ only requires the mirror symmetries; it follows
from a result of Carlotto, Schulz, and Wiygul \cite{CSW25a, CSW25c}. There is an upper bound due to Lima \cite[Thm.~A]{Lim22}, which is the free boundary form of the comparison between area and energy done by Ejiri--Micallef for closed surfaces \cite{EM08}. It bounds the index by the energy index plus the dimension of the Teichm\"uller
space. Seo \cite[Thm.~1.1]{Seo24} shows that every $\Sigma\in\cS_b$ has first Steklov eigenvalue equal to one (this was recently extended by Chodosh and Gianocca \cite{CG26} to every embedded FBMS of genus zero); consequently, Lima's theorem gives the upper bound $\Ind(\Sig)\le3b-3$ (see Lemma \ref{lem:energy}), which leaves a gap for $b\ge4$.  


\end{remark}

\begin{remark} We refine the upper bound in two ways. First, using representation theory, we prove an upper bound for each symmetry
type separately (Proposition~\ref{prop:EM}). 
Second, using circle domains, we compute the Teichm\"uller
directions of each type: they are induced by the motions of $b$ round caps modulo the M\"obius motions
(Proposition~\ref{prop:teich}).
\end{remark}

Precisely, the proof of Theorem~\ref{thm:main} follows the steps below.
\begin{itemize}[leftmargin=*,itemsep=2pt]
\item \emph{Step 1: the lower bound} (Proposition~\ref{prop:lower}). The bound $\Ind(\Sig)\ge2b$ is due to Carlotto,
Schulz, and Wiygul \cite[Prop.~4.2]{CSW25a}, \cite[proof of Cor.~3.8]{CSW25c}, via a nodal domain argument going back to
\cite{Choe90,MR91}. We further record that the negative space contains a copy of $(\cR_b\ominus A_2)^+\oplus A_1^-$.
\item \emph{Step 2: area and energy} (Proposition~\ref{prop:EM}). We prove the refined, equivariant comparison
described above. Since $\sigma_1=1$ \cite[Thm.~1.1]{Seo24} (see also \cite{CG26}), the negative space of the energy is controlled by the representation
$\R^3$.
\item \emph{Step 3: Teichm\"uller directions} (Proposition~\ref{prop:teich}). Inspired by the work of
Choe, Fraser, and Schoen \cite{ChFS25}, who parametrize symmetric FBMS of genus zero in
rectangular prisms by circle domains, we realize $\Sig$ conformally as the sphere minus $b$ equal round caps centered on
the equator (Lemma~\ref{lem:model}). The Teichm\"uller directions are then induced by the motions of the caps modulo the M\"obius
motions, and, with the measurable Riemann mapping theorem, we bound them for each symmetry type,
whereas the Riemann--Roch argument of \cite{Lim22} only counts their dimension. To our knowledge, this is the first use
of circle domains to bound Teichm\"uller directions by symmetry type, and the first in a computation of the Morse index.
This settles the $\sigma_h$-even sector completely and, in the odd sector, excludes the types $A_1^-$ and $E_1^-$ beyond
the translation.
\item \emph{Step 4: the odd sector} (Proposition~\ref{prop:odd}). We combine the decomposition of symmetric types, the order of an umbilic point, and Courant's nodal domain theorem to analyze the index form restricted to the upper half of $\Sig$  

\item \emph{Step 5: the nullity} (Proposition~\ref{prop:null}). We use the non-strict version of the inequality in Step~2,
together with the fact that the Steklov eigenfunctions of $\Sig$ with eigenvalue one are the linear combinations of the
coordinate functions (Lemma~\ref{lem:steklov}, a special case of \cite[Props.~5.12, 5.15]{KKMS24}; compare
\cite{KM24}). It follows that the elements of $\ker S$ are exactly the normal parts of the rotations.
\end{itemize}

The combination of symmetry reduction and Courant's nodal domain theorem is crucial in the computation of the index
and nullity; see, for example,  \cite{KW20} for the Lawson surfaces $\xi_{g,1}$ in $\mathbb S^3$. Our approach also yields lower bounds for the index of FBMS with symmetry, which might be of independent interest. The first one concerns rotational symmetry.

\begin{theorem}\label{thm:rot}
Let $\Sig\subset\Ball$ be an FBMS of genus zero with $b\ge2$ boundary components. Let $\rho$ be a
rotation of order $m\ge3$ about a line $\ell$ through the origin such that $\rho(\Sig)=\Sig$ and $\nu\circ\rho=\rho\,\nu$;
the latter holds automatically if $m$ is odd or if $\langle X,\nu\rangle>0$ in the interior of $\Sig$. Let $d$ be the
number of boundary components mapped to themselves by $\rho$. Then $d\le2$, $d\equiv b\pmod m$, and
\[
  \Ind(\Sig)\ \ge\ 1+(2-d)(m-1).
\]
In particular, $\Ind(\Sig)\ge2m-1$ if $m$ divides $b$.
\end{theorem}

\begin{remark}\label{rem:fibre}
Theorem~\ref{thm:rot} follows from a more general lower bound without any symmetry assumption
(Proposition~\ref{prop:fibre}). 
\end{remark}

Next, we consider symmetry groups which act irreducibly on $\R^3$.

\begin{theorem}\label{thm:poly}
Let $\Sig\subset\Ball$ be an FBMS, and let $\Gamma\subset O(3)$ be a finite group acting
irreducibly on $\R^3$ such that $g(\Sig)=\Sig$ for all $g\in\Gamma$. Then $\Ind(\Sig)\ge7$.
\end{theorem}

The finite subgroups of $O(3)$ acting irreducibly on $\R^3$ are the seven polyhedral groups $T$, $T_d$, $T_h$, $O$,
$O_h$, $I$, and $I_h$ \cite[Ch.~2]{GB85}. Accordingly, we say that $\Sig$ has \emph{polyhedral symmetry} if it is invariant under one of
them. Also, the proof of Theorem~\ref{thm:poly} is based on the conformal vector fields of the ball. 


\begin{remark}
\label{cor:four}
It has been conjectured that the critical catenoid is the only FBMS in $\Ball$ of index four. The author partially resolved that conjecture for embedded FBMS with two boundary components and showed that, in general, such a surface is a polar graph and 
has genus zero \cite{Tra20}. Thus, the embedded case of the conjecture remains open precisely for surfaces of genus zero with $3\leq b\leq 13$ (the upper bound is due to \cite{ACS18}). The results in this paper imply that

\begin{enumerate}[label=\textup{(\roman*)},itemsep=2pt]
\item $\Sig$ has no polyhedral symmetry (from Theorem~\ref{thm:poly});
\item if $\Sig$ has a rotational symmetry of order $m\ge3$, then either
      $b\equiv2\pmod m$, or $m\in\{3,4\}$ and $b\equiv1\pmod m$. In particular, $m\nmid b$, and there is no such
      symmetry for $b=3$ (from Theorem~\ref{thm:rot}).
\end{enumerate}
\end{remark}

\subsection*{Organization}
The organization of the paper is as follows. Section~\ref{sec:prelim} collects preliminary results which will be used
throughout; in particular, it recalls the representation theory of the dihedral and prismatic groups
(\S\ref{ssec:reps}) and the lower bound of Carlotto, Schulz, and Wiygul (Theorem~\ref{thm:mirrors}), which
Proposition~\ref{prop:lower} refines type by type. In Section~\ref{sec:rot}, we establish the lower bound in terms of the fibers of the Gauss map (Proposition~\ref{prop:fibre}) and prove
Theorem~\ref{thm:rot}. Section~\ref{sec:poly} is devoted to M\"obius boosts and the
proof of Theorem~\ref{thm:poly}. Finally, Theorem~\ref{thm:main} is proved in
Section~\ref{sec:prism} following the steps outlined above.

\medskip

\noindent\textbf{Acknowledgements.} The idea of using representation theory of finite groups came to the author during the workshop \emph{The Bochner Technique} at the American Institute of Mathematics (AIM) in May 2026. The author
would like to thank AIM for its hospitality.
\medskip

\noindent\textbf{Declaration of AI usage:} 
Claude Opus 5.5 and Fable 5 were used to perform
literature review, place different mechanisms into one framework, carry out computation, generate illustrations, and as adversarial readers. Fable 5 was particularly useful: when asked about Teichm\"uller theory for FBMS, it brought up circle domains from \cite{ChFS25} and assisted with adapting their framework to our situation (Lemma \ref{lem:model} and Theorem \ref{thm:surj}). All proofs are completely written by the author, who retains exclusive responsibility for the content of this work. 

\section{Preliminaries}\label{sec:prelim}

In this section, we fix our notation and collect preliminary results which will be used throughout the paper. Most of
them are well-known, and short proofs are included for completeness.

\subsection{Conventions}
Throughout the paper, $\Sig\subset\Ball$ denotes a compact, connected, properly embedded FBMS.
It is real-analytic up to its boundary \cite{DHT10}. We write $X$ for the position vector, $\nu$ for a unit normal, $\eta$
for the outward unit conormal along $\bdSig$, and $\tau$ for a unit tangent vector of $\bdSig$. Moreover, $\Sigma^\circ=\Sigma\setminus\partial\Sigma$ denotes the interior of $\Sigma$, $L$ the length of $\bdSig$ and $|\Sig|$ the area of $\Sig$. The second fundamental form is given by
$A(v,w)=\langle D_vw,\nu\rangle$; thus, $\langle d\nu(v),w\rangle=-A(v,w)$. Since $\Sig$ is minimal, $\tr A=0$. The free
boundary condition reads
\begin{equation}\label{eq:fb}
  \eta=X\ \text{ on }\bdSig;\qquad\text{consequently,}\qquad \zeta:=\langle X,\nu\rangle=0\ \text{ on }\bdSig .
\end{equation}
For $a\in\R^3$, we put $X_a=\langle a,X\rangle$, $\nu_a=\langle a,\nu\rangle$, and $R_a=\langle a\times X,\nu\rangle$.
Here $R_a$ is the normal component of the Killing field $x\mapsto a\times x$, which generates the rotations about $\R a$.
Also, $v^\top$ denotes the tangential part of a vector $v$.

The index form $S$ of \eqref{eq:S} is defined on $H^1(\Sig)$. Let $J=\Delta+|A|^2$ be the Jacobi operator. Integrating by
parts, we have, for $f\in H^1(\Sig)$ and $h$ smooth,
\begin{equation}\label{eq:Sparts}
  S(f,h)=-\int_\Sig f\,Jh\dA+\int_{\bdSig}f\,(\partial_\eta h-h)\ds .
\end{equation}
Relative to $L^2(\Sig)$, the form $S$ has a discrete spectrum which is bounded from below, and its eigenfunctions are
smooth up to $\bdSig$. Recall that $\Ind(\Sig)$ is the number of negative eigenvalues of $S$; equivalently, it is the largest
dimension of a subspace of $H^1(\Sig)$ on which $S$ is negative definite. Moreover, $\ker S$ consists of the solutions of
$Jh=0$ with $\partial_\eta h=h$ on $\bdSig$, and its dimension is the nullity $\Null(\Sig)$.

The following identities are standard.

\begin{lemma}\label{lem:basic}
Let $k=A(\tau,\tau)$ along $\bdSig$, and let $'$ denote the derivative with respect to arclength. Then the following
statements hold:
\begin{enumerate}[label=\textup{(\roman*)},itemsep=2pt]
\item $A(\eta,\tau)=0$, $d\nu(\eta)=k\eta$, and $\tau'=-\eta+k\nu$ along $\bdSig$;
\item $\Delta X=0$ in $\Sig$ and $\partial_\eta X=X$ on $\bdSig$; furthermore, $\int_{\bdSig}X\ds=0$, $L=2|\Sig|$, and
      $\int_{\bdSig}X_a^2\ds=\int_\Sig\bigl(|a|^2-\nu_a^2\bigr)\dA$;
\item $J\nu_a=0$ and $\partial_\eta\nu_a=kX_a$; $J\zeta=0$ and $\partial_\eta\zeta=k$; $JR_a=0$ and $\partial_\eta R_a=R_a$;
      in particular, $R_a\in\ker S$;
\item $S(\nu_a,\nu_a)=-2\int_\Sig\nu_a^2\dA$.
\end{enumerate}
\end{lemma}

\begin{proof}
Direct computation. See also \cite{Tra20, Tra21}. 	
\end{proof}

\emph{The Steklov problem.} The Steklov eigenvalues $0=\sigma_0<\sigma_1(\Sig)\le\sigma_2(\Sig)\le\cdots$ of $\Sig$ are
the numbers $\sigma$ for which there is a function $u\not\equiv0$ with $\Delta u=0$ in $\Sig$ and $\partial_\eta u=\sigma u$
on $\bdSig$. The first non-zero eigenvalue is characterized by
\[
\sigma_1(\Sig)=\inf\Big\{\frac{\int_\Sig|\nabla u|^2\,dA}{\int_{\bdSig}u^2\,ds}:u\in H^1(\Sig),\ \int_{\bdSig}u\,ds=0,\
u|_{\bdSig}\not\equiv0\Big\},
\]
and the infimum is attained exactly by the eigenfunctions with eigenvalue $\sigma_1(\Sig)$. By Lemma~\ref{lem:basic}(ii),
the coordinate functions $X_1,X_2,X_3$ are Steklov eigenfunctions with eigenvalue one and boundary mean zero; in
particular, $\sigma_1(\Sig)\le1$.
\subsection{Symmetries}\label{ssec:sym}
A \emph{symmetry} of $\Sig$ is an element $g\in O(3)$ with $g(\Sig)=\Sig$. We have the following observation; see also \cite{CSW25a}.

\begin{lemma}\label{lem:sym}
Let $g$ be a symmetry of $\Sig$. Then the following statements hold:
\begin{enumerate}[label=\textup{(\roman*)},itemsep=2pt]
\item $\nu\circ g=\eps(g)\,g\nu$ for some $\eps(g)\in\{\pm1\}$, and $\eps$ is a homomorphism on the group of symmetries
      of $\Sig$;
\item if $g$ is the reflection in a plane and $\Sig$ is not a flat disc, then $\eps(g)=1$;
\item if $\eps(g)=1$, then $f\mapsto g\cdot f:=f\circ g^{-1}$ preserves $S$ and $\ker S$; $g$ preserves the orientation
      of $\Sig$ if and only if $\det g=1$; and, for $a\in\R^3$,
      \[
        g^*A=A,\qquad g\cdot\nu_a=\nu_{ga},\qquad g\cdot R_a=\det(g)R_{ga}.
      \]
\end{enumerate}
\end{lemma}

\begin{proof}
	(i) Since $dg$ maps $T_x\Sig$ onto $T_{gx}\Sig$, we have $\nu(gx)=\pm g\nu(x)$ with a locally constant sign. As $\Sig$ is
connected, the sign is constant. The homomorphism property is clear.

(ii) Let $g$ be the reflection in a plane $P$. Since $g$ exchanges the two sides of $P$ and $\Sig^\circ$ is connected,
$\Sig^\circ\cap P\ne\emptyset$. If $\Sig$ were tangent to $P$ at a point of $\Sig^\circ\cap P$, then near that point $\Sig$
would be the graph over $P$ of a function $w$ with $w=-w$; by analyticity, $\Sig$ would be contained in $P$ and, hence,
be a flat disc. Therefore, at every $p\in\Sig^\circ\cap P$, the tangent plane $T_p\Sig$ is a $g$-invariant plane
different from $P$. Such a plane contains the normal of $P$; thus, $\nu(p)\in P$ and $\nu(gp)=\nu(p)=g\nu(p)$. That is,
$\eps(g)=1$.

(iii) If $\nu\circ g=g\nu$, then $g$ restricts to an isometry of $\Sig$ with $g^*A=A$, which implies the invariance of $S$
and $\ker S$. Moreover, $(dg\,v,dg\,w,\nu\circ g)$ has the orientation of $(v,w,\nu)$ multiplied by $\det g$. Finally,
for $x\in\Sig$, we have $(g\cdot\nu_a)(x)=\langle a,g^{-1}\nu(x)\rangle=\nu_{ga}(x)$ and
\[
  (g\cdot R_a)(x)=\langle a\times g^{-1}x,g^{-1}\nu(x)\rangle=\langle g(a\times g^{-1}x),\nu(x)\rangle
  =\det(g)\langle ga\times x,\nu(x)\rangle .\qedhere
\]

\end{proof}

In a local conformal coordinate $z$ on $\Sig$, the Hopf differential $\mathfrak h=A(\partial_z,\partial_z)\,dz^2$, where
$A$ is extended complex-bilinearly, is holomorphic. It vanishes identically only if $\Sig$ is a flat disc. Its zeros in
$\Sig^\circ$ are the umbilics of $\Sig$, and the order $m_q\ge0$ of the zero at $q$ is called the \emph{order of $q$ as
an umbilic}; here $m_q=0$ if $A(q)\ne0$.

\begin{lemma}\label{lem:hopf}
Suppose that $\Sig$ is not a flat disc. Let $\rho$ be a rotation of order $m\ge3$ about the line $\R a$, $|a|=1$, such that
$\rho(\Sig)=\Sig$ and $\nu\circ\rho=\rho\nu$. If $q\in\Sig^\circ$ is fixed by $\rho$, then $T_q\Sig=a^\perp$,
$\nu(q)=\pm a$, $X(q)\in\R a$, and $m_q\ge m-2$.
\end{lemma}

\begin{proof} 
		Since $\rho(q)=q$, the point $q$ lies on $\R a$, and $T_q\Sig$ is a $\rho$-invariant plane; as $m\ge3$, it is $a^\perp$.
		Hence $\nu(q)=\pm a$. The bound $m_q\ge m-2$ is \cite[Lemma~E.2]{CSW25c}.
\end{proof}
\subsection{Representations}\label{ssec:reps}
Here we review basic facts from the representation theory of finite groups. The main reference is \cite{Ser77}.

Let $H$ be a finite group. A (real) \emph{representation} of $H$, or an \emph{$H$-module}, is a finite-dimensional real
vector space $V$ together with a linear action of $H$ (equivalently, a homomorphism $H\to GL(V)$); as usual, we denote the representation and the space by the same
letter. Two representations are isomorphic if there is an $H$-equivariant linear isomorphism between them. The
\emph{character} of $V$ is the function on $H$ given by
\[
  \chi_V(\gamma)=\tr(\gamma|_V).
\]
It is constant on conjugacy classes and satisfies $\chi_V(1)=\dim V$. Furthermore, the characters of $V\oplus W$ and
$V\otimes W$ are $\chi_V+\chi_W$ and $\chi_V\chi_W$, respectively \cite[\S2.1]{Ser77}. If $\chi$ is a one-dimensional representation,
then $V\otimes\chi$ is the space $V$ on which $\gamma$ acts by $\chi(\gamma)\gamma$.

A representation $\pi$ is called
\emph{irreducible} if $\pi\ne0$ and its only invariant subspaces are $0$ and $\pi$, and \emph{absolutely irreducible} if it
remains irreducible after complexification. It turns out that all irreducible representations of $D_b$ and of $\Dbh$,
which are described below, are absolutely irreducible. Thus, the results of \cite{Ser77}, which are formulated over $\C$,
apply to them. Every representation is a direct sum of irreducible ones, and we denote by $m_\pi(V)$ the
multiplicity of an irreducible representation $\pi$ in $V$. For
absolutely irreducible representations $\pi$ and $\pi'$ of $H$, the following facts are well known:
\begin{enumerate}[label=(\alph*),itemsep=2pt]
	\item \emph{Schur's lemma.} If $\pi\not\cong\pi'$, then an $H$-equivariant linear map $\pi\to\pi'$ vanishes
	\cite[\S2.2]{Ser77}. 
	\item \emph{Multiplicities.} By \cite[\S2.3, Thm.~4]{Ser77}, for every representation $V$ of $H$,
	\[
	  m_\pi(V)=\frac1{|H|}\sum_{\gamma\in H}\chi_V(\gamma)\chi_\pi(\gamma).
	\]
	In particular, if all irreducible representations of $H$ are absolutely irreducible, then two representations of $H$
	with the same character are isomorphic.
	\item \emph{Permutation representations.} Suppose that $H$ acts on a finite set $Y$. Then it acts on the space of real
	functions $\phi$ on $Y$ by $(\gamma\cdot\phi)(y)=\phi(\gamma^{-1}y)$. Since $\gamma$ permutes the indicator functions of
	the points of $Y$, the character of this representation at $\gamma$ is the number of points of $Y$ fixed by $\gamma$.
	For $Y=H$, on which $H$ acts by left multiplication, we obtain the \emph{regular representation}. Its character is
	$|H|$ at the identity and $0$ elsewhere, since an element $\gamma\ne1$ fixes no point of $H$. Therefore, by (b), it
	contains every absolutely irreducible representation $\pi$ with multiplicity $\dim\pi$ \cite[\S2.4]{Ser77}.
\end{enumerate}

\emph{The dihedral group.} The group $D_b$ has order $2b$; it consists of the rotations $\Rot^j$, $0\le j<b$, and the $b$
reflections in the vertical mirror planes which all share an axis $\ell$ and are equally spaced at angles $\pi/b$. Its irreducible representations are listed in Table~\ref{tab:chars} (see
\cite[\S5.3]{Ser77}). Precisely, they are the trivial representation $A_1$; the representation $A_2$ in which $\Rot$ acts
trivially and the reflections act by $-1$; for even $b$, two further one-dimensional representations $B_1$ and $B_2$ in
which $\Rot$ acts by $-1$; and the two-dimensional representations $E_k$, $1\le k<b/2$, in which $\Rot$ acts as the
rotation by $2\pi k/b$ and the reflections act as reflections. 
We denote the regular
representation of $D_b$ by $\cR_b$; it has dimension $2b$. By (b), (c), and Table~\ref{tab:chars}, we have
\begin{align*}
\cR_b &\cong A_1\oplus A_2\oplus\bigoplus_{1\le k<b/2}E_k^{\oplus2}\quad(b\text{ odd}),\\
\cR_b &\cong A_1\oplus A_2\oplus B_1\oplus B_2\oplus\bigoplus_{1\le k<b/2}E_k^{\oplus2}\quad(b\text{ even}),\\
\cR_b &\cong A_2\otimes\cR_b \qquad\text{and}\qquad E_1 \cong A_2\otimes E_1.
\end{align*}

\begin{table}[ht]
	\centering
	\begin{tabular}{lccc}
		\hline
		$\pi$ & $\dim\pi$ & $\chi_\pi(\Rot^j)$ & $\chi_\pi(s)$, $s$ a reflection\\
		\hline
		$A_1$ & $1$ & $1$ & $1$\\
		$A_2$ & $1$ & $1$ & $-1$\\
		$B_1$, $B_2$ ($b$ even) & $1$ & $(-1)^j$ & $\pm1$\\
		$E_k$, $1\le k<b/2$ & $2$ & $2\cos(2\pi kj/b)$ & $0$\\
		\hline
	\end{tabular}
	\medskip
	\caption{The characters of the irreducible representations of $D_b$. For even $b$, the reflections form two conjugacy
		classes, on which $B_1$ and $B_2$ take the values $1$ and $-1$ in the two possible ways.}
	\label{tab:chars}
\end{table}

\emph{The prismatic group.} For a representation $\pi$ of $D_b$, we denote by $\pi^\pm$ the
representation of $\Dbh=D_b\times\langle\sigma_h\rangle$ which restricts to $\pi$ on $D_b$ and on which $\sigma_h$ acts by
$\pm1$. That is, $\pi^+$ and $\pi^-$ are the $\sigma_h$-even and $\sigma_h$-odd versions of $\pi$. 
Since $\sigma_h$ commutes with every element of $D_b$, the irreducible representations of $\Dbh$ are the $\pi^\pm$ with $\pi$
irreducible \cite[\S3.2]{Ser77}, and they are absolutely irreducible. For instance, the defining (also called standard) representation of $\Dbh$ on $\R^3$, given by the inclusion
$\Dbh\subset O(3)$, is
\[
  \R^3\cong E_1^+\oplus A_1^- .
\]
Indeed, the group $D_b$ acts on the horizontal plane $\ell^\perp$ through $E_1$ and fixes the vertical axis $\ell$
pointwise, while $\sigma_h$ fixes $\ell^\perp$ pointwise and reverses $\ell$. We write $\det$ for the
one-dimensional representation $\gamma\mapsto\det\gamma\in\{\pm1\}$ of $\Dbh$. The rotations about $\ell$ have
determinant $1$, and the reflections, including $\sigma_h$, have determinant $-1$; hence, $\det\cong A_2^-$. Similarly, the rotation fields $a\times p$ on $\Sph$ satisfy $\gamma_*(a\times p)=\det(\gamma)\,\gamma a\times p$, so that
$a\mapsto a\times p$ identifies them with
\[
  \R^3\otimes\det\cong(E_1\otimes A_2)^-\oplus(A_1\otimes A_2)^+\cong E_1^-\oplus A_2^+.
\]

\emph{Isotypic parts and types.} For an irreducible representation $\pi$ of $H$, the \emph{$\pi$-isotypic part} of $V$
is the sum of all subrepresentations of $V$ isomorphic to $\pi$. Every representation is the direct
sum of its isotypic parts \cite[\S2.6]{Ser77}. We say that a subrepresentation or an element of $V$ \emph{has type
	$\pi$} if it lies in the $\pi$-isotypic part. If $\pi$ is one-dimensional, with character $\chi_\pi$, then an element
$v$ has type $\pi$ exactly when $\gamma\cdot v=\chi_\pi(\gamma)v$ for all $\gamma\in H$, that is, when $\R v$ is invariant
and $H$ acts on it through $\pi$. For instance, a function $f$ on $\Sig$ has type $A_2^+$ if $\Rot\cdot f=f$, $s\cdot f=-f$
for every vertical reflection $s$, and $\sigma_h\cdot f=f$.

\emph{The boundary components.} For $\Sig\in\cS_b$, let $\cP_b$ be the permutation representation of $D_b$ on the set of
boundary components of $\Sig$. 
By (c), its character at $\gamma$ is the number of boundary components fixed by $\gamma$. Since $\Rot$ permutes them
cyclically, $D_b$ acts transitively, and the stabilizer of a boundary component has order $2b/b=2$; it consists of the identity and one reflection. Hence, the
character of $\cP_b$ is $b$ at the identity and vanishes on $\Rot^j$ for $0<j<b$. 

Moreover, if a reflection $s$ preserves
a boundary component $\Gamma$, then $s\Rot^ms=\Rot^{-m}$ shows that $s$ preserves $\Rot^m\Gamma$ exactly when
$\Rot^{2m}\Gamma=\Gamma$, that is, when $b$ divides $2m$. Consequently, for odd $b$, every reflection preserves exactly one
boundary component. For even $b$, the reflections preserving a boundary component form one conjugacy class of $b/2$
reflections, each of which preserves exactly two boundary components, while the other $b/2$ reflections preserve none;
in this case, we label $B_1$ and $B_2$ so that $B_1$ takes the value $1$ on the first class. In both cases, since characters are constant
on conjugacy classes, the sum of $\chi_{\cP_b}\chi_\pi$ over the reflections equals $b\,\chi_\pi(s)$, where $s$ is a
reflection preserving a boundary component. Hence, (b) gives
\[
  m_\pi(\cP_b)=\frac1{2b}\bigl(b\dim\pi+b\,\chi_\pi(s)\bigr)=\frac12\bigl(\dim\pi+\chi_\pi(s)\bigr),
\]
and Table~\ref{tab:chars} yields
\begin{align*}
\cP_b &\cong A_1\oplus\bigoplus_{1\le k<b/2}E_k\quad(b\text{ odd}),\\
\cP_b &\cong A_1\oplus B_1\oplus\bigoplus_{1\le k<b/2}E_k\quad(b\text{ even}).
\end{align*}
In particular,
\begin{equation}\label{eq:Pb}
	m_{A_1}(\cP_b)=1\qquad\text{and}\qquad m_{E_1}(\cP_b)=1.
\end{equation}

\subsection{Symmetry types and nodal sets}\label{ssec:types}
In this subsection, we recall how a group of symmetries decomposes the index. Let $G$ be a finite group acting
orthogonally on $L^2(\Sig)$ and preserving $H^1(\Sig)$ and $S$; by Lemma~\ref{lem:sym}(iii), this holds for every
finite group of symmetries $g$ of $\Sig$ with $\nu\circ g=g\nu$, acting by $g\cdot f=f\circ g^{-1}$. 

Let $\mathcal N$
be the negative eigenspace of $S$, that is, the span of the eigenfunctions with negative eigenvalues; it is
$G$-invariant. 
By the discussion in \S\ref{ssec:reps}, for $\ind_\pi(S):=m_\pi(\mathcal N)$ and $\operatorname{null}_\pi(S)=m_\pi(\ker S)$, we have
\[\Ind(\Sig)=\dim\mathcal N=\sum_\pi\ind_\pi(S)\dim\pi \text{ and } \Null(\Sig)=\sum_\pi\operatorname{null}_\pi(S)\dim\pi.\]

The following observation will be used repeatedly to produce negative directions; compare \cite[Lemma 12]{MR91}  and \cite[Prop. 3.1(i)]{CSW25a}.

\begin{lemma}\label{lem:iso}
Let $V\subset H^1(\Sig)$ be a finite-dimensional $G$-invariant subspace on which $S\le0$. Then, for every irreducible
representation $\pi$ of $G$,
\[
  \ind_\pi(S)\ge m_\pi(V)-m_\pi(V\cap\ker S).
\]
In particular, $\Ind(\Sig)\ge\dim V-\dim(V\cap\ker S)$.
\end{lemma}

\begin{proof}
Let $\mathcal N^\perp$ be the $L^2$-orthogonal complement of $\mathcal N$ in $H^1(\Sig)$. It is $G$-invariant,
$S(\mathcal N,\mathcal N^\perp)=0$, and $S\ge0$ on $\mathcal N^\perp$. Let $p:V\to\mathcal N$ be the projection along
$\mathcal N^\perp$; it is $G$-equivariant. We claim that $\ker p\subset V\cap\ker S$. 

Indeed, let $v\in\ker p$. Then
$v\in\mathcal N^\perp$ and, thus, $S(v,v)\ge0$. On the other hand, $S(v,v)\le0$ as $v\in V$. Hence, $v$ minimizes $S$ on
$\mathcal N^\perp$, which implies $S(v,\mathcal N^\perp)=0$. Together with $S(v,\mathcal N)=0$, it gives $v\in\ker S$.
Consequently,
\[
  m_\pi(V)-m_\pi(V\cap\ker S)\le m_\pi(V)-m_\pi(\ker p)=m_\pi(p(V))\le m_\pi(\mathcal N).\qedhere
\]
\end{proof}

\emph{Nodal sets.} We will use the following local facts for a solution $u\not\equiv0$ of $\Delta u+cu=0$, with $c$ smooth, on a connected
open subset of $\Sig^\circ$. First, $u$ cannot vanish on an open set, nor to infinite order at a point \cite{Aro57}.
Second, if $u(p)=0$, then, in a conformal coordinate $z$ centered at $p$,
\[
  u=\operatorname{Re}(\alpha z^j)+O(|z|^{j+1})
\]
for some $j\ge1$ and $\alpha\in\C\setminus\{0\}$ \cite{Bers55}. Third, near $p$, the zero set of $u$ consists of $2j$ arcs
issuing from $p$ which are pairwise disjoint away from $p$ and tangent at $p$ to the $2j$ rays on which
$\operatorname{Re}(\alpha z^j)=0$, and $u$ has alternating signs in the $2j$ sectors between consecutive arcs
\cite{Cheng76}.
\begin{lemma}\label{lem:pieces}
	Let $u\in\ker S$, and let $\Om_1,\dots,\Om_N\subset\Sig$ be pairwise disjoint relatively open sets whose union has full
	measure. Suppose that $u$ vanishes on the boundary of each $\Om_i$ relative to $\Sig$ and that $u\not\equiv0$ on each
	$\Om_i$. Then the functions $u_i=u\,\mathbf 1_{\Om_i}$ lie in $H^1(\Sig)$, $S(u_i,u_k)=0$ for all $i,k$, and the span $V$
	of $u_1,\dots,u_N$ satisfies $V\cap\ker S=\R u$. In particular, $\Ind(\Sig)\ge N-1$.
\end{lemma}

\begin{proof}
	For $\delta>0$, let $\phi_\delta(t)=\operatorname{sign}(t)\max\{|t|-\delta,0\}$. Since $|u|<\delta$ near the relative
	boundary of $\Om_i$, the function $\phi_\delta(u)\mathbf 1_{\Om_i}$ is Lipschitz. Moreover, it converges to $u_i$ in
	$H^1(\Sig)$ as $\delta\to0$, because $\nabla u=0$ almost everywhere on $\{u=0\}$. Thus, $u_i\in H^1(\Sig)$. For $i\ne k$,
	we have $u_iu_k=0$ on $\Sig$, including on $\bdSig$, and $\langle\nabla u_i,\nabla u_k\rangle=0$ almost everywhere; hence,
	$S(u_i,u_k)=0$. Since
	$u\in\ker S$ and $u=\sum_ku_k$, it follows that
	\[
	S(u_i,u_i)=S(u,u_i)-\sum_{k\ne i}S(u_k,u_i)=0 .
	\]
	Next, suppose that $w=\sum_ic_iu_i\in\ker S$. Then $Jw=0$, and $w$ agrees with $c_1u$ on the non-empty open set
	$\Om_1\cap\Sig^\circ$. By unique continuation \cite{Aro57}, $w=c_1u$; since $u\not\equiv0$ on each $\Om_i$, all $c_i$ are
	equal. Finally, the $u_i$ are linearly independent, so $\dim V=N$, and Lemma~\ref{lem:iso} gives $\Ind(\Sig)\ge N-1$.
\end{proof}

Finally, we recall the lower bounds of Carlotto, Schulz, and Wiygul for surfaces with mirror symmetries. 
The first bound is due to \cite[Prop.~4.2]{CSW25a}; the second follows by combining parts
(b) and (c) of that proposition, as in \cite[proof of Cor.~3.8]{CSW25c}.


\begin{theorem}[Carlotto--Schulz--Wiygul]\label{thm:mirrors}
Let $\Sig\subset\Ball$ be an FBMS other than a flat disc. Suppose that $\Sig$ is invariant under
the reflections in $n\ge2$ distinct planes containing a line $\ell$ through the origin and that $\Sig$ is not
rotationally symmetric about $\ell$. Then $\Ind(\Sig)\ge2n-1$. Furthermore, if $\Sig$ is also invariant under the
reflection in the plane $\ell^\perp$, then $\Ind(\Sig)\ge2n$.
\end{theorem}

\subsection{Circle domains}
A \emph{circle domain} is the complement $\Om\subset\Sph$ of finitely many pairwise disjoint closed round discs, which
are called its \emph{caps}. We will use Koebe's theorem \cite{Koe20}: every planar domain of finite connectivity is
conformally equivalent to a circle domain, and every conformal map between circle domains is the restriction of a
M\"obius transformation (see also \cite[Ch.~V]{Gol69}). The following is an equivariant version; compare \cite[Lemma~4.1 and \S4]{ChFS25}. 

\begin{lemma}\label{lem:koebe}
Let $\Sig$ have genus zero and $b\ge2$ boundary components, and let $G$ be a finite group of symmetries of $\Sig$ such that
$\nu\circ g=g\nu$ for all $g\in G$. Then there exist a circle domain $\Om\subset\Sph$ with $b$ caps, a diffeomorphism $F$
of $\Sig$ onto $\overline\Om$ which is conformal with respect to the induced metric of $\Sig$, and an injective
homomorphism $G\to O(3)$, $g\mapsto\hat g$, such that $\det\hat g=\det g$ and $F\circ g=\hat g\circ F$ on $\Sig$ for all
$g\in G$.
\end{lemma}

\begin{proof}
Gluing a disc to each boundary component of $\Sig$ and uniformizing the resulting sphere, we see that $\Sig^\circ$ is
conformally equivalent to a planar domain of connectivity $b$. Thus, Koebe's theorem gives a conformal map $F_0$ of
$\Sig^\circ$ onto a circle domain. By Carath\'eodory's theorem \cite{Gol69}, which applies to finitely connected domains bounded by Jordan curves, $F_0$ extends to a homeomorphism of $\Sig$ onto the closure of that domain. 

Near $p\in\bdSig$, the surface $\Sig$ extends across $\bdSig$ as a real-analytic minimal surface; hence there is a conformal coordinate $\zeta$ in which $\Sig$ corresponds to a closed upper half-disc. Composing $F_0$ with a M\"obius transformation mapping the boundary circle through $F_0(p)$ onto the real line, we obtain a map which is holomorphic in the open half-disc, continuous up to the real segment, and real
there. By the Schwarz reflection principle, it extends to an injective holomorphic map near $\zeta=0$, whose derivative
therefore does not vanish. Thus, $F_0$ extends to a diffeomorphism of $\Sig$ onto the closure of that domain.

For $g\in G$, the map $F_0gF_0^{-1}$ is conformal if $\det g=1$ and anticonformal if $\det g=-1$ by
Lemma~\ref{lem:sym}(iii). Composing it with complex conjugation in the second case and applying the uniqueness part of
Koebe's theorem, we see that it is the restriction of a M\"obius transformation, respectively of an anti-M\"obius
transformation. These transformations form a finite group of conformal and anticonformal transformations of $\Sph$. It
is isomorphic to $G$; indeed, only the identity of $O(3)$ fixes $\Sig$ pointwise, as $\Sig$ is not planar for $b\ge2$. 

By Poincar\'e extension \cite[Ch.~3]{Bea83}, every M\"obius or anti-M\"obius transformation of $\Sph$ extends uniquely to
an isometry of hyperbolic $3$-space; thus, this group
acts on hyperbolic $3$-space by isometries. A finite group of isometries of hyperbolic space fixes a point, for instance,
the circumcenter of an orbit \cite[Cor.~II.2.8]{BH99}. Conjugating by a M\"obius transformation which moves this point to
the center of the ball, the group becomes a subgroup of the stabilizer of the center, which is $O(3)$. Since the image of a circle domain
under a M\"obius transformation is again a circle domain, this gives $F$ and $g\mapsto\hat g$.
\end{proof}

\begin{figure}[ht]
	\centering
	\begin{tikzpicture}[>=stealth]
		\def\R{2.5}
		\draw (0,0) circle (\R);
		\draw[dashed,gray] (\R,0) arc[start angle=0,end angle=180,x radius=\R,y radius=0.714];
		\draw (-\R,0) arc[start angle=180,end angle=360,x radius=\R,y radius=0.714];
		\filldraw[fill=gray!12,draw=gray,dashed] (0,0.714) ellipse (0.43 and 0.23);
		\filldraw[fill=gray!25] (-2.38,0) ellipse (0.13 and 0.43);
		\filldraw[fill=gray!25] (2.38,0) ellipse (0.13 and 0.43);
		\filldraw[fill=gray!25] (0,-0.714) ellipse (0.54 and 0.5);
		\draw[densely dashed] (0,-0.714) ellipse (0.75 and 0.71);
		\draw[->,thick] (0.82,-0.40) .. controls (1.1,-0.42) .. (1.5,-0.48);
		\draw[->,thick] (0,-1.47) -- (0,-2.05);
		\fill (0,2.40) circle (1.5pt) node[above] {$F(N)$};
		\draw[gray,very thin] (1.55,-0.48) -- (3.0,-0.48) node[right,black] {\small longitude};
		\draw[gray,very thin] (0.53,-1.21) -- (3.0,-1.21) node[right,black] {\small radius};
		\draw[gray,very thin] (0.06,-1.85) -- (3.0,-1.85) node[right,black] {\small latitude};
	\end{tikzpicture}
	\caption{The circle-domain model of a surface in $\cS_b$ for $b=4$: the sphere minus four equal round caps centered
		on the equator. A cap can change its radius, its longitude, or its latitude. $N$ is an umbilic point from Lemma \ref{lem:star}}
	\label{fig:caps}
\end{figure}
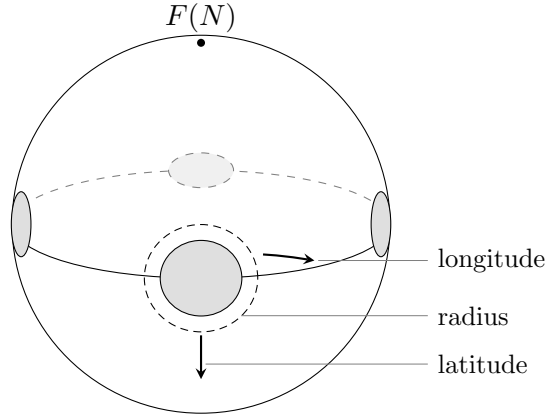

\begin{remark}\label{rem:circle}

Choe, Fraser, and Schoen \cite{ChFS25} use this model to construct FBMS of genus zero with six
boundary components in a rectangular prism, that is, a product of three intervals, which is a special case of the products of balls considered in \cite{ChFS25}. For their symmetry group, generated by the reflections in the coordinate
planes, the model is bounded by three pairs of circles centered on the coordinate axes, and its moduli space is
parametrized by the three radii \cite[\S4]{ChFS25}. To each point of the moduli space they associate a harmonic map into
the prism, and they differentiate its energy in the radii by the flow of a vector field which moves one pair of circles;
the derivative is the pairing of the stress--energy tensor with the Lie derivative of the round metric along this field
\cite[\S6]{ChFS25}. A critical point of the energy yields a conformal harmonic map and, hence, an FBMS
 \cite[Thm.~6.7]{ChFS25}.

In the comparison between area and energy in our setting
(Proposition~\ref{prop:EM}), a normal variation $f\nu$ can be completed by a tangential field to an infinitesimally
conformal variation exactly when the Teichm\"uller direction vanishes; thus these directions bound the difference
between the index of the area and that of the energy. We therefore use all motions of the circles, not only the
invariant ones. For $\Sig\in\cS_b$, the model of Lemma~\ref{lem:model} consists of $b$ equal caps centered on the equator
(Figure~\ref{fig:caps}); the motions of the caps, modulo the M\"obius motions, exhaust the Teichm\"uller directions
(Lemma~\ref{lem:psi} and Theorem~\ref{thm:surj}); and they decompose explicitly under $D_{bh}$ (Proposition~\ref{prop:teich}
and Remark~\ref{rem:caps}). The invariant motions reduce to a change of the common radius, which is the analogue of the
moduli space of \cite{ChFS25}.
\end{remark}

\section{The Gauss map and rotations}\label{sec:rot}

In this section, we prove Theorem~\ref{thm:rot}. The main negative directions are obtained by cutting a Jacobi field $R_a$ along curves on which it vanishes
(Lemma~\ref{lem:pieces}). 

\subsection{The fibers of the Gauss map}\label{ssec:fibre}
We first compute the order of vanishing of $R_a$ at a point where the normal is parallel to $a$.

\begin{lemma}\label{lem:orders}
Suppose that $\Sig$ is not a flat disc. Let $q\in\Sig^\circ$ and $a=\nu(q)$. Then $R_a$ vanishes at $q$ to order exactly
$m_q+1+\eps_q$, where $\eps_q=1$ if $X(q)\in\R a$, that is, if $q$ is a critical point of $|X|^2$ on $\Sig$, and $\eps_q=0$
otherwise.
\end{lemma}

\begin{proof}
After a rotation, we may assume that $a=-e_3$. Near $q$, we use the Weierstrass representation in a conformal
coordinate $z$ centered at $q$:
\[
  X=q+\operatorname{Re}\int_0^z\Bigl(\tfrac12\mathfrak f(1-\mathfrak g^2),\tfrac i2\mathfrak f(1+\mathfrak g^2),\mathfrak f\mathfrak g\Bigr)dz,\qquad
  \nu=\frac{(2\operatorname{Re}\mathfrak g,\,2\operatorname{Im}\mathfrak g,\,|\mathfrak g|^2-1)}{1+|\mathfrak g|^2},
\]
where $\mathfrak g(0)=0$ and, after a complex rescaling of $z$, $\mathfrak f(0)=2$. The Hopf differential is a non-zero
constant multiple of $\mathfrak f\mathfrak g'\,dz^2$; thus, $\mathfrak g=\alpha z^{m+1}+O(|z|^{m+2})$ with $m=m_q$ and
$\alpha\ne0$.

We identify the horizontal plane $T_q\Sig$ with $\C$. Then the horizontal part of $X-q$ is $\bar z+O(|z|^2)$, its
vertical part is $O(|z|^2)$, the horizontal part of $\nu-a$ is $2\mathfrak g+O(|z|^{3m+3})$, and $a\times w=-iw$ for
horizontal $w$. Hence, for horizontal $c\ne0$,
\[
  \nu_c=\langle c,\nu-a\rangle=2\operatorname{Re}\bigl(\bar c\,\alpha z^{m+1}\bigr)+O(|z|^{m+2})
\]
vanishes at $q$ to order exactly $m+1$. Next, let $u_q=\langle a\times(X-q),\nu\rangle$ be the normal component of the
rotations about the normal line $q+\R a$. Since $a\times(X-q)$ is horizontal, we compute
\[
  u_q=\langle a\times(X-q),\nu-a\rangle=\operatorname{Re}\bigl(-i\bar z\cdot2\bar\alpha\bar z^{m+1}\bigr)+O(|z|^{m+3})
  =-2\operatorname{Im}\bigl(\alpha z^{m+2}\bigr)+O(|z|^{m+3}).
\]
Thus, $u_q$ vanishes at $q$ to order exactly $m+2$. Finally, $R_a=u_q+\nu_{a\times X(q)}$, where $a\times X(q)\perp a$,
and $a\times X(q)=0$ if and only if $X(q)\in\R a$. Since $\nabla|X|^2=2X^\top$, and $X^\top(q)=0$ if and only if
$X(q)\in\R\nu(q)$, the result then follows.
\end{proof}

\begin{proposition}\label{prop:fibre}
Let $\Sig$ be an FBMS of genus $g$ other than a flat disc, and let $a\in\Sph$ be such that $\Sig$
is not rotationally symmetric about $\R a$. Then
\[
  \Ind(\Sig)\ \ge\ 1-2g+\sum_{q\in\Sig^\circ,\ \nu(q)=\pm a}\bigl(m_q+\eps_q\bigr),
\]
where $\eps_q$ is as in Lemma~\ref{lem:orders}.
\end{proposition}

\begin{proof}
First, $R_a\not\equiv0$; otherwise, the Killing field of the rotations about $\R a$ would be tangent to $\Sig$, and these
rotations would preserve $\Sig$. Let $Z=R_a^{-1}(0)$, and let $N$ be the number of components of $\Sig\setminus Z$, the
nodal domains of $R_a$. By Lemma~\ref{lem:basic}(iii) and Lemma~\ref{lem:pieces}, $\Ind(\Sig)\ge N-1$. Thus, it suffices
to bound $N$ from below.

\emph{Step 1: the local structure of the zero set.} With the orientation $\tau=\nu\times\eta$ of $\bdSig$, we have $X\times\nu=-\tau$
on $\bdSig$. Hence, $R_a=-\langle a,\tau\rangle$ there; that is, on each boundary component, $R_a$ is minus the derivative
of $X_a$ with respect to arclength. Therefore, $R_a$ vanishes at least twice on each boundary component. Moreover, it does
not vanish on any arc of $\bdSig$; otherwise, $R_a$ and $\partial_\eta R_a=R_a$ would vanish there, and $R_a\equiv0$ by
unique continuation. Since $\Sig$ is real-analytic up to $\bdSig$, it extends across $\bdSig$ to a slightly larger
minimal surface on which $R_a$ still satisfies $JR_a=0$. Thus, the local description of nodal sets in \S\ref{ssec:types} also applies
at the points of $\bdSig$.

Let $p\in\bdSig$ be a zero of $R_a$ of order $j$, and choose a conformal coordinate $z$ centered at $p$ in which $\bdSig$
is the real axis and $\Sig$ lies in the closed upper half plane. Then $R_a=\operatorname{Re}(\alpha z^j)+O(|z|^{j+1})$, and
$\alpha$ is real because $\partial_\eta R_a=R_a$ vanishes to order $j$ at $p$. The zero set of $\operatorname{Re}(z^j)$
consists of $2j$ rays, none of which is horizontal, and $j$ of them lie in the upper half plane. Consequently, near $p$,
the set $Z$ consists of $j$ arcs entering $\Sig^\circ$. Similarly, at an interior zero of order $j$, the set $Z$ consists
of $2j$ arcs. By compactness, $Z$ has finitely many zeros of order at least two in $\Sig^\circ$ and finitely many points
on $\bdSig$.

\emph{Step 2: counting.} Let $b$ be the number of boundary components of $\Sig$, and let $\hat\Sig$ be the closed surface
of genus $g$ obtained by gluing a disc to each of them. We consider the finite graph $\mathcal Z=Z\cup\bdSig$ in $\hat\Sig$ whose vertices are the interior
zeros of order at least two, the boundary zeros, and one additional vertex on each closed component of $Z$ without zeros of
order at least two. The faces of $\mathcal Z$ are the $N$ nodal domains and the $b$ discs. Each face $F$ is a connected
open surface of finite type; thus, $\chi(F)\le1$, and
\[
  N+b\ \ge\ \sum_F\chi(F)=\chi(\hat\Sig)-\chi(\mathcal Z)=2-2g+E-V ,
\]
where $V$ and $E$ denote the numbers of vertices and edges. Moreover, $E-V=\sum_v\bigl(\frac12\deg v-1\bigr)$. An interior
vertex which is a zero of order $j$ has degree $2j$ and contributes $j-1$. A boundary vertex with $j$ arcs entering
$\Sig^\circ$ has degree $j+2$ and contributes $\frac j2\ge\frac12$; there are at least $2b$ boundary vertices. The
additional vertices contribute $0$. Consequently,
\[
  N\ \ge\ 2-2g+\sum_{p\in Z\cap\Sig^\circ}\bigl(\operatorname{ord}_pR_a-1\bigr).
\]
The points $q\in\Sig^\circ$ with $\nu(q)=\pm a$ are zeros of $R_a$. By Lemma~\ref{lem:orders}, applied with $\nu(q)$ in
place of $a$ and using $R_{-a}=-R_a$, each of them contributes $m_q+\eps_q$. Therefore,
\[
  \Ind(\Sig)\ \ge\ N-1\ \ge\ 1-2g+\sum_{q\in\Sig^\circ,\ \nu(q)=\pm a}\bigl(m_q+\eps_q\bigr).\qedhere
\]
\end{proof}

\subsection{Rotational symmetry}\label{ssec:rotation}
We are now ready to prove Theorem~\ref{thm:rot}. 

\begin{proof}[Proof of Theorem~\ref{thm:rot}]
First, we verify the assertion that $\nu\circ\rho=\rho\nu$ holds automatically if $m$ is odd or if $\zeta>0$ in
$\Sig^\circ$. If $m$ is odd, then
$\eps(\rho)=\eps(\rho)^m=\eps(\rho^m)=1$ by Lemma~\ref{lem:sym}(i). If $\zeta>0$ in $\Sig^\circ$, then
$\zeta\circ\rho=\eps(\rho)\zeta$ forces $\eps(\rho)=1$.

Let $a$ be a unit vector spanning $\ell$. By Lemma~\ref{lem:koebe} applied to $G=\langle\rho\rangle$, there exist a circle
domain $\Om$, a conformal diffeomorphism $F:\Sig\to\overline\Om$, and a rotation $\hat\rho\in SO(3)$ of order $m$ such that
$F\circ\rho=\hat\rho\circ F$. All non-trivial powers of $\hat\rho$ have the same two fixed points on $\Sph$, and all other
orbits of $\hat\rho$ consist of $m$ points.

We claim that each cap of $\Om$ is either $\hat\rho$-invariant, with a fixed point of $\hat\rho$ in its interior, or belongs
to an orbit of $m$ caps. Indeed, a non-trivial rotation preserving a round disc fixes its center and has no fixed point on its
boundary circle. Suppose that a cap $D$ is mapped to itself by $\hat\rho^i\ne\mathrm{id}$. Then the center of $D$ is a fixed
point of $\hat\rho^i$ and, hence, of $\hat\rho$. Consequently, $\hat\rho(D)\cap D\ne\emptyset$ and, since the caps are
pairwise disjoint, $\hat\rho(D)=D$. By the same fact, the fixed point of $\hat\rho$ in $D$ lies in its interior. This proves
the claim.

The invariant caps correspond to the $d$ boundary components mapped to themselves by $\rho$. Thus, $d\le2$ and $m$ divides
$b-d$. Furthermore, $\hat\rho$, hence $\rho$, has exactly $2-d$ fixed points on $\Sig$, all of them in $\Sig^\circ$. If $d=2$, then the claimed bound reads $\Ind(\Sig)\ge1$, which holds since $S(1,1)=-\int_\Sig|A|^2-L<0$. Let $d\le1$. Then
$\Sig$ is not rotationally symmetric about $\ell$; otherwise, every boundary component would be invariant under $\rho$,
and $b=d\le1$. By Lemma~\ref{lem:hopf}, each of the $2-d$ fixed points $q$ of $\rho$ satisfies $\nu(q)=\pm a$,
$X(q)\in\R a$, and $m_q\ge m-2$; in particular, $\eps_q=1$. Proposition~\ref{prop:fibre} with $g=0$ gives
\[
  \Ind(\Sig)\ \ge\ 1+(2-d)(m-1).
\]
Finally, if $m$ divides $b$, then $m$ divides $d$. Since $d\le2<m$, it follows that $d=0$ and, hence, $\Ind(\Sig)\ge2m-1$.
\end{proof}


\section{M\"obius boosts and polyhedral symmetry}\label{sec:poly}

In this section, we prove Theorem~\ref{thm:poly}. Towards that end, we test the index form on the constant function, the
coordinate functions, and the normal parts of the conformal vector fields of the ball which generate the M\"obius boosts.
Let $\mathbf e=(1,X_1,X_2,X_3)^{\mathrm T}$, $V_0=\Span\{1,X_1,X_2,X_3\}$, and $\delta_0=(1,0,0,0)^{\mathrm T}$.

\begin{lemma}\label{lem:V0}
Suppose that $\Sig$ is not a flat disc. Then $S$ is negative definite on $V_0$, and its Gram matrix with respect to
$1,X_1,X_2,X_3$ is $-\mathbf G$, where
\[
  \mathbf G=\int_\Sig|A|^2\,\mathbf e\,\mathbf e^{\mathrm T}\dA+L\,\delta_0\delta_0^{\mathrm T}.
\]
\end{lemma}

\begin{proof}
By \eqref{eq:S}, \[S(1,1)=-\int_\Sig|A|^2-L.\] 
By Lemma~\ref{lem:basic}(ii), \[\int_\Sig\langle\nabla f,\nabla X_i\rangle=\int_{\bdSig}fX_i\]
for every $f\in H^1(\Sig)$; thus, $S(f,X_i)=-\int_\Sig|A|^2fX_i$. This gives the Gram matrix. For
$w\in\R^4\setminus\{0\}$,
\[
  w^{\mathrm T}\mathbf Gw=\int_\Sig|A|^2(w^{\mathrm T}\mathbf e)^2\dA+Lw_0^2
\]
vanishes only if $w_0=0$ and $\langle(w_1,w_2,w_3),X\rangle=0$ on the dense set $\{|A|\ne0\}$, that is, only if $\Sig$ is
contained in a plane. Hence, $\mathbf G$ is positive definite.
\end{proof}

For $a\in\R^3$, the vector field $W_a(x)=\frac12(1+|x|^2)a-\langle a,x\rangle x$ generates the M\"obius transformations of
$\Ball$ which move the origin in the direction $a$; it is tangent to $\Sph$. Its normal component on $\Sig$ is
\[
  u_a=\langle W_a,\nu\rangle=\tfrac12\bigl(1+|X|^2\bigr)\nu_a-X_a\,\zeta .
\]

\begin{proposition}\label{prop:boost}
Let $a\in\R^3$. Then $Ju_a=2\nu_a$ in $\Sig$, and $u_a=\nu_a$ and $\partial_\eta u_a=u_a$ on $\bdSig$. Consequently,
\begin{equation}\label{eq:boostid}
  S(u_a,f)=-2\int_\Sig\nu_a\,f\dA\qquad\text{for all }f\in H^1(\Sig).
\end{equation}
Furthermore, for $a,c\in\R^3$, we have $S(u_a,u_c)=\langle\mathbf Qa,c\rangle$, where
\[
  \mathbf Q=6\int_\Sig X\otimes X\dA-\Bigl(|\Sig|+\int_\Sig|X|^2\dA\Bigr)I .
\]
Moreover, the area of $\Sig$ satisfies \[|\Sig|=\int_\Sig\bigl(2|X|^2-\zeta^2\bigr)\dA.\]
\end{proposition}

\begin{proof}
On a minimal surface, $\Delta|X|^2=4$, $\nabla|X|^2=2X^\top$, and $\nabla X_a=a^\top$. Since $d\nu$ is self-adjoint,
$\nabla\nu_a=d\nu(a^\top)$ and $\nabla\zeta=d\nu(X^\top)$. Using $J(fh)=fJh+h\Delta f+2\langle\nabla f,\nabla h\rangle$
and Lemma~\ref{lem:basic}(iii), we compute
\[
  J\bigl(\tfrac12(1+|X|^2)\nu_a\bigr)=2\nu_a+2\langle X^\top,d\nu(a^\top)\rangle\qquad\text{and}\qquad
  J(X_a\zeta)=2\langle a^\top,d\nu(X^\top)\rangle .
\]
Thus, $Ju_a=2\nu_a$. On $\bdSig$, we have $|X|=1$, $\zeta=0$, $\partial_\eta|X|^2=2$, $\partial_\eta\nu_a=kX_a$, and
$\partial_\eta\zeta=k$. Hence, $u_a=\nu_a$ and
\[
  \partial_\eta u_a=\nu_a+kX_a-kX_a=u_a .
\]
Now \eqref{eq:boostid} follows from \eqref{eq:Sparts} and an approximation argument.

For the second claim, \eqref{eq:boostid} gives
\[
  S(u_a,u_c)=-\int_\Sig(1+|X|^2)\nu_a\nu_c\dA+2\int_\Sig\zeta\nu_aX_c\dA .
\]
By the symmetry of $S$, we may replace $2\zeta\nu_aX_c$ by $\zeta(X_a\nu_c+X_c\nu_a)$. On a minimal surface,
$\operatorname{div}_\Sig(X_aX_cX^\top)=4X_aX_c-\zeta(X_a\nu_c+X_c\nu_a)$, and $\langle X^\top,\eta\rangle=1$ on $\bdSig$.
Moreover, $\int_{\bdSig}X_aX_c=\int_\Sig\langle a^\top,c^\top\rangle$ as in the proof of Lemma~\ref{lem:basic}(ii). Hence,
\begin{equation}\label{eq:zetamix}
  \int_\Sig\zeta\bigl(X_a\nu_c+X_c\nu_a\bigr)\dA=4\int_\Sig X_aX_c\dA-\int_\Sig\bigl(\langle a,c\rangle-\nu_a\nu_c\bigr)\dA .
\end{equation}
Next, Green's formula for $|X|^2$ and $X_aX_c$, with $\Delta(X_aX_c)=2\langle a^\top,c^\top\rangle$,
$\partial_\eta(X_aX_c)=2X_aX_c$, and $\partial_\eta|X|^2=2$ on $\bdSig$, gives
\[
  \int_\Sig|X|^2\bigl(\langle a,c\rangle-\nu_a\nu_c\bigr)\dA=2\int_\Sig X_aX_c\dA .
\]
Substituting, we obtain
\[
  S(u_a,u_c)=6\int_\Sig X_aX_c\dA-\langle a,c\rangle\Bigl(|\Sig|+\int_\Sig|X|^2\dA\Bigr).
\]
Finally, taking the trace of \eqref{eq:zetamix} over an orthonormal basis gives
\[2\int_\Sig\zeta^2=4\int_\Sig|X|^2-2|\Sig|.\] 
\end{proof}

Consequently, 
the following
criterion holds.

\begin{proposition}\label{prop:boostcount}
Suppose that $\Sig$ is not a flat disc. Let $\boldsymbol\lambda=\int_\Sig\mathbf e\,\nu^{\mathrm T}\dA\in\R^{4\times3}$ and
$\mathbf C=\mathbf Q+4\boldsymbol\lambda^{\mathrm T}\mathbf G^{-1}\boldsymbol\lambda$. Then
\[
  \Ind(\Sig)\ge4+n_-(\mathbf C),
\]
where $n_-$ denotes the number of negative eigenvalues.
\end{proposition}

\begin{proof}
By Lemma~\ref{lem:V0} and Proposition~\ref{prop:boost}, the Gram matrix of $S$ with respect to the seven functions
$1,X_1,X_2,X_3,u_{e_1},u_{e_2},u_{e_3}$ is
\[
  M=\begin{pmatrix}-\mathbf G&-2\boldsymbol\lambda\\-2\boldsymbol\lambda^{\mathrm T}&\mathbf Q\end{pmatrix}.
\]
By Haynsworth's inertia additivity \cite{Hay68},
\[
  n_-(M)=n_-(-\mathbf G)+n_-\bigl(\mathbf Q-4\boldsymbol\lambda^{\mathrm T}(-\mathbf G)^{-1}\boldsymbol\lambda\bigr)=4+n_-(\mathbf C).
\]
On the other hand, the number of negative eigenvalues of the Gram matrix of a finite family equals the index of $S$ on
its span, which is at most $\Ind(\Sig)$. The result then follows.
\end{proof}

The following integral inequality will control the coupling term.

\begin{lemma}\label{lem:zeta}
For every FBMS $\Sig$,
\[
  \Bigl(\int_\Sig\zeta\dA\Bigr)^2\ \le\ \frac1{12}\int_\Sig|X^\top|^2\dA\cdot\int_\Sig|A|^2|X|^2\dA .
\]
\end{lemma}

\begin{proof}
Since $\operatorname{div}X^\top=2$ and $\nabla\zeta=d\nu(X^\top)$, we have
\[
  \operatorname{div}(\zeta X^\top)=\langle d\nu(X^\top),X^\top\rangle+2\zeta .
\]
As $\zeta=0$ on $\bdSig$, integrating yields $2\int_\Sig\zeta=-\int_\Sig\langle d\nu(X^\top),X^\top\rangle$. The eigenvalues of
$d\nu$ are $\pm|A|/\sqrt2$; thus, $|\langle d\nu(v),v\rangle|\le|A||v|^2/\sqrt2$. By the Cauchy--Schwarz inequality,
\begin{equation}\label{eq:zeta1}
  4\Bigl(\int_\Sig\zeta\dA\Bigr)^2\le\frac12\int_\Sig|A|^2|X^\top|^2\dA\cdot\int_\Sig|X^\top|^2\dA .
\end{equation}

Next, $|\nabla\zeta|^2=|d\nu(X^\top)|^2=\frac12|A|^2|X^\top|^2$. Multiplying $\Delta\zeta=-|A|^2\zeta$ by $\zeta$ and
integrating, using $\zeta=0$ on $\bdSig$, we obtain
\[
  \frac12\int_\Sig|A|^2|X^\top|^2\dA=\int_\Sig|\nabla\zeta|^2\dA=\int_\Sig|A|^2\zeta^2\dA .
\]
Since $|X^\top|^2=|X|^2-\zeta^2$, it follows that $\int_\Sig|A|^2|X|^2=3\int_\Sig|A|^2\zeta^2$. Consequently,
\[
  \int_\Sig|A|^2|X^\top|^2\dA=\frac23\int_\Sig|A|^2|X|^2\dA .
\]
Substituting into \eqref{eq:zeta1}, we get the desired claim.
\end{proof}

\begin{proof}[Proof of Theorem~\ref{thm:poly}] We divide the argument into several steps.

\emph{Step 1: the group.} 
The rotation subgroup $\Gamma_0=\Gamma\cap SO(3)$ has index at most two, and, by the classification of the finite subgroups
of $SO(3)$ \cite[Ch.~2]{GB85}, it is cyclic, dihedral, or one of the rotation groups $T$, $O$, and $I$ of a regular tetrahedron, octahedron, and icosahedron. We claim that $\Gamma_0\in\{T,O,I\}$.

Indeed, if $\Gamma_0$ is non-trivial and cyclic, or dihedral of order at least six, then $\Gamma$ preserves its axis,
respectively its principal axis. If $\Gamma_0$ is the Klein four-group, then $\Gamma$ permutes its three axes. An element
permuting them cyclically would have order divisible by three, and a suitable power of it would be an element of order
three, hence of determinant one, in $\Gamma_0$, which is impossible; so $\Gamma$ preserves one of the axes. If $\Gamma_0$ is trivial,
then $\Gamma$ has order at most two. In all these cases, $\Gamma$ acts reducibly, which proves the claim. 

Consequently,
$\Gamma$ contains a copy of $T$, which acts irreducibly on $\R^3$. By Lemma~\ref{lem:sym}(i), $\eps$ restricts to a
homomorphism $T\to\{\pm1\}$, which is trivial because $T$ is isomorphic to the alternating group on four letters with no subgroup of index two. Thus, $\nu\circ g=g\nu$ for $g\in T$.

\emph{Step 2: the integral identities.} By Lemma~\ref{lem:sym}(iii), we have $\nu_a\circ g^{-1}=\nu_{ga}$ and $|A|^2\circ g^{-1}=|A|^2$
for $g\in T$; moreover, $X_a\circ g^{-1}=X_{ga}$. Hence, $\int_\Sig\nu$ and $\int_\Sig|A|^2X$ are $T$-invariant vectors and,
thus, vanish. Furthermore, the matrices $\int_\Sig X\otimes X$, $\int_\Sig|A|^2X\otimes X$, and $P=\int_\Sig X\otimes\nu$
commute with $T$. Since the commutant of an irreducible real representation of odd dimension is $\R$, they are scalar. As
$\tr P=\int_\Sig\zeta$, we obtain
\[
  P=\frac13\Bigl(\int_\Sig\zeta\dA\Bigr)I,\qquad \int_\Sig X\otimes X\dA=\frac13\int_\Sig|X|^2\dA\cdot I,\qquad
  \int_\Sig|A|^2X\otimes X\dA=\beta I,
\]
where $\beta=\frac13\int_\Sig|A|^2|X|^2>0$.

\emph{Step 3: the boost form.} In Proposition~\ref{prop:boostcount}, we have
$\boldsymbol\lambda=\bigl(\begin{smallmatrix}0\\P\end{smallmatrix}\bigr)$ and $\mathbf G=\mathrm{diag}(\mathbf G_{00},\beta I)$.
Thus, by Lemma~\ref{lem:zeta},
\[
  4\boldsymbol\lambda^{\mathrm T}\mathbf G^{-1}\boldsymbol\lambda=\frac4\beta P^{\mathrm T}P
  =\frac43\,\frac{(\int_\Sig\zeta)^2}{\int_\Sig|A|^2|X|^2}\,I\ \le\ \frac19\int_\Sig|X^\top|^2\dA\cdot I .
\]
By Proposition~\ref{prop:boost}, $\mathbf Q=\bigl(2\int_\Sig|X|^2-|\Sig|-\int_\Sig|X|^2\bigr)I=-\int_\Sig|X^\top|^2\,I$.
Hence,
\[
  \mathbf C\le-\frac89\int_\Sig|X^\top|^2\dA\cdot I .
\]
Here $\int_\Sig|X^\top|^2>0$; otherwise, $X^\top\equiv0$ would place $\Sig$ in a sphere about the origin, which is
impossible for a minimal surface. Therefore, $n_-(\mathbf C)=3$, and Proposition~\ref{prop:boostcount} gives
${\Ind(\Sig)\ge7}$.
\end{proof}


\section{The index of prismatic surfaces}\label{sec:prism}

In this section, we prove Theorem~\ref{thm:main}. Throughout, $b\ge3$,
$\Sig\in\cS_b$, $\ell=\R e_3$, and $G=\Dbh=D_b\times\langle\sigma_h\rangle$. 
As $G$ is generated by reflections, Lemma~\ref{lem:sym} gives $\nu\circ\gamma=\gamma\nu$ for all
$\gamma\in G$. Consequently, $G$ acts on functions by $\gamma\cdot f=f\circ\gamma^{-1}$, preserving $S$ and $\ker S$.

\subsection{The upper half}\label{ssec:star}
In this subsection, we describe the fundamental domain $\overline U$ of $\sigma_h$ and the fixed point $N$ of $\Rot$ in it; see Figure~\ref{fig:upper} for the case $b=5$.

\begin{lemma}\label{lem:star}
	Let $\Lambda=\Sig\cap\{x_3=0\}$, $\overline U=\Sig\cap\{x_3\ge0\}$, and $U=\Sig^\circ\cap\{x_3>0\}$. Then the following
	statements hold:
	\begin{enumerate}[label=\textup{(\roman*)},itemsep=2pt]
		\item $\Lambda$ consists of $b$ disjoint smooth arcs, each of which meets $\bdSig$ orthogonally at its two endpoints;
		\item $\overline U$ is a closed disc with interior $U$, and its boundary consists of $\Lambda$ and the arcs
		$\bdSig\cap\{x_3\ge0\}$;
		\item $\Rot$ has exactly one fixed point $N$ in $U$; moreover, $N$ lies on $\ell$ and is fixed by $D_b$, it is an umbilic of
		order at least $b-2$, and $T_N\Sig=\ell^\perp$, so that $D_b$ acts on $T_N\Sig$ as on $\ell^\perp$, that is, through
		the representation $E_1$.
	\end{enumerate}
\end{lemma}

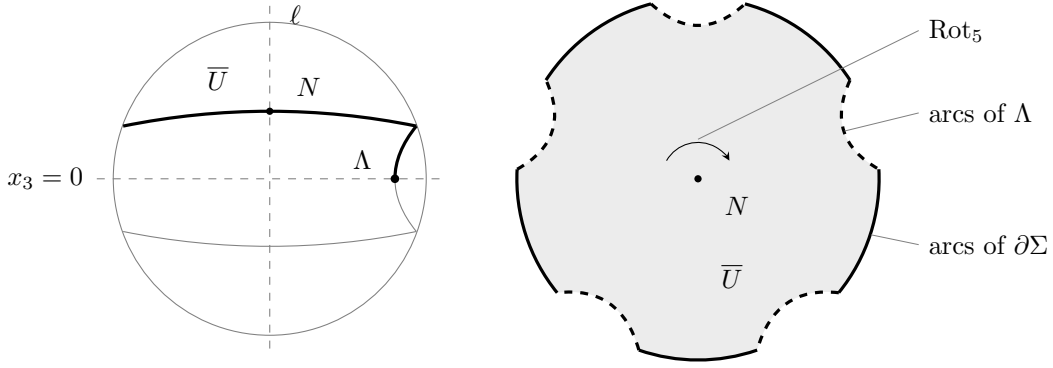
\begin{figure}[ht]
	\centering
	\begin{tikzpicture}[x=1pt,y=1pt,yscale=-1,scale=0.62,every node/.style={font=\small}]
		\draw[gray,thin] (140,200) circle (95);
		\draw[gray,dashed] (35,200) -- (245,200);
		\draw[gray,dashed] (140,95) -- (140,305);
		\draw[very thick] svg {M 51 168 Q 140 150 229 168};
		\draw[very thick] svg {M 229 168 Q 216 184 216 200};
		\draw[gray] svg {M 216 200 Q 216 216 229 232};
		\draw[gray] svg {M 51 232 Q 140 250 229 232};
		\fill (216,200) circle (2.6);
		\fill (140,159) circle (2.2);
		\node[anchor=base west] at (150,150) {$N$};
		\node[anchor=base west] at (96,146) {$\overline U$};
		\node[anchor=base east] at (208,194) {$\Lambda$};
		\node[anchor=base west] at (146,106) {$\ell$};
		\node[anchor=east] at (33,200) {$x_3=0$};
		\fill[gray!15] svg {M 428.5 93.8 A 110 110 0 0 1 492.3 140.1 Q 476.1 175.3 509.8 194.2 A 110 110 0 0 1 485.5 269.2 Q 447.0 264.7 435.9 304.0 A 110 110 0 0 1 364.1 304.0 Q 353.0 264.7 314.5 269.2 A 110 110 0 0 1 290.2 194.2 Q 323.9 175.3 307.7 140.1 A 110 110 0 0 1 371.5 93.8 Q 400 120 428.5 93.8 Z};
		\draw[very thick] svg {M 428.5 93.8 A 110 110 0 0 1 492.3 140.1};
		\draw[very thick] svg {M 509.8 194.2 A 110 110 0 0 1 485.5 269.2};
		\draw[very thick] svg {M 435.9 304.0 A 110 110 0 0 1 364.1 304.0};
		\draw[very thick] svg {M 314.5 269.2 A 110 110 0 0 1 290.2 194.2};
		\draw[very thick] svg {M 307.7 140.1 A 110 110 0 0 1 371.5 93.8};
		\draw[very thick,dashed] svg {M 492.3 140.1 Q 476.1 175.3 509.8 194.2};
		\draw[very thick,dashed] svg {M 485.5 269.2 Q 447.0 264.7 435.9 304.0};
		\draw[very thick,dashed] svg {M 364.1 304.0 Q 353.0 264.7 314.5 269.2};
		\draw[very thick,dashed] svg {M 290.2 194.2 Q 323.9 175.3 307.7 140.1};
		\draw[very thick,dashed] svg {M 371.5 93.8 Q 400 120 428.5 93.8};
		\draw[->,>=stealth] svg {M 380.9 189 A 22 22 0 0 1 419.1 189};
		\fill (400,200) circle (2.2);
		\node[anchor=base west] at (410,222) {$N$};
		\node[anchor=base west] at (408,266) {$\overline U$};
		\draw[gray,very thin] (400,176) -- (534,110) node[anchor=west,black] {$\mathrm{Rot}_5$};
		\draw[gray,very thin] (489,171) -- (534,160) node[anchor=west,black] {arcs of $\Lambda$};
		\draw[gray,very thin] (505,234) -- (534,240) node[anchor=west,black] {arcs of $\partial\Sigma$};
	\end{tikzpicture}
	\caption{The upper half for $b=5$, drawn schematically after the $b$-noid. Left: the section through a vertical mirror plane, which meets a bridge on one side of $\ell$ and a
		boundary component on the other; $\overline U$ is drawn thick. Right: $\overline U$ seen from above. Its boundary
		alternates between the five arcs of $\Lambda$ (dashed) and the five arcs of $\bdSig\cap\{x_3\ge0\}$ (solid), and
		$\mathrm{Rot}_5$ turns $\overline U$ about its unique fixed point $N$.}
	\label{fig:upper}
\end{figure}

\begin{proof}
(i) At $x\in\Lambda$, we have $\sigma_hx=x$. Thus, $\nu(x)=\nu(\sigma_hx)=\sigma_h\nu(x)$ is horizontal, and
$e_3\in T_x\Sig$. Each boundary component $\Gamma$ is $\sigma_h$-invariant; hence, it meets the equator, since otherwise
$\Gamma$ and $\sigma_h(\Gamma)$ would lie in opposite open hemispheres. Let $p\in\Gamma\cap\Lambda$. By
Lemma~\ref{lem:sym}(iii), $\sigma_h$ reverses the orientation of $\Sig$, so $d\sigma_h$ is a reflection of $T_p\Sig$. It
fixes $\eta(p)=p$; therefore, the tangent line of $\Gamma$ at $p$ is spanned by $e_3$. Consequently, $\Sig$ and $\bdSig$
are transversal to $\{x_3=0\}$, and $\Lambda$ is a compact smooth curve whose endpoints are its points on $\bdSig$, where
it is tangent to $\eta$. Moreover, $\sigma_h$ reverses the orientation of $\Gamma$, so it has exactly two fixed points on
$\Gamma$. Thus, $\Lambda$ has $2b$ endpoints, and it consists of $b$ arcs and possibly some closed curves; the latter are
excluded in (ii).

(ii) The set $\overline U$ is a compact connected orientable surface with corners. Indeed, it is a fundamental domain for
$\sigma_h$ with $\overline U\cap\sigma_h(\overline U)=\Lambda$; it is homeomorphic to $\Sig/\sigma_h$ and, hence, connected;
and it lies in $\Sig$. Moreover,
\[
  2-b=\chi(\Sig)=2\chi(\overline U)-\chi(\Lambda)=2\chi(\overline U)-b .
\]
Thus, $\chi(\overline U)=1$, and $\overline U$ is a closed disc. Its boundary is a single closed curve containing $\Lambda$;
therefore, $\Lambda$ has no closed components, and the interior of $\overline U$ is $U$.

(iii) The rotation $\Rot$ preserves $\Sig$ and the coordinate $x_3$; thus, it preserves $\overline U$. We claim that $\Rot$ has
no fixed point on
$\partial\overline U$. 

Indeed, $\Rot$ moves every boundary component of $\Sig$. If it fixed a point of an arc $\alpha$ of
$\Lambda$, then $\Rot(\alpha)=\alpha$, and $\Rot$ would preserve the set of at most two boundary components containing the
endpoints of $\alpha$. As $\Rot$ fixes no boundary component, it would exchange two of them, and $\Rot^2$ would fix one, which
is impossible for a cyclic permutation of $b\ge3$ objects. This proves the claim. By Brouwer's theorem, $\Rot$ has a fixed
point $N\in U$, and the fixed points of $\Rot$ in $\R^3$ form $\ell$. Moreover, $N$ is the only fixed point of $\Rot$ in $U$,
because $\Rot$ restricts to a conformal automorphism of finite order $\ne1$ of the disc $U$. Every element of $D_b$ fixes
$\ell$ pointwise; hence, it fixes $N$ and acts on $T_N\Sig$ by its linear part. Finally, by Lemma~\ref{lem:hopf},
$T_N\Sig=\ell^\perp$ and $m_N\ge b-2$.
\end{proof}

The next lemma is the key topological input for the odd sector and for the nullity.

\begin{lemma}\label{lem:jordan}
Let $f\not\equiv0$ be a smooth function on $\overline U$ such that $\Delta f+cf=0$ in $U$ for a smooth function $c$.
Suppose that the sets $U_+=\{x\in U:f(x)>0\}$ and $U_-=\{x\in U:f(x)<0\}$ are connected. Then $f(N)\ne0$ or $df_N\ne0$.
\end{lemma}

\begin{proof}
We argue by contradiction. Suppose that $f(N)=0$ and $df_N=0$. By the facts on nodal sets in \S\ref{ssec:types}, $f$ vanishes at $N$ to a finite
order $j\ge2$. In a conformal coordinate $z$ centered at $N$, we have $f=\operatorname{Re}(\alpha z^j)+O(|z|^{j+1})$ with
$\alpha\ne0$, and near $N$ the zero set of $f$ consists of $2j\ge4$ arcs tangent to the rays on which
$\operatorname{Re}(\alpha z^j)=0$. Let $\Theta_1,\dots,\Theta_{2j}$ be the sectors between consecutive arcs, in cyclic order,
with $f>0$ on $\Theta_1$. Then $\Theta_1\cup\Theta_3\subset U_+$ and $\Theta_2\cup\Theta_4\subset U_-$.

For small $\delta>0$, consider the segments $\beta_1$ and $\beta_3$ of length $\delta$ from $N$ along the rays bisecting the
sectors of $\{\operatorname{Re}(\alpha z^j)>0\}$ which correspond to $\Theta_1$ and $\Theta_3$. They lie in
$\Theta_1\cup\{N\}$ and $\Theta_3\cup\{N\}$, respectively. We join their endpoints by a simple arc in $U_+$ and pass to the
part of this arc between its last point on $\beta_1$ and the next point on $\beta_3$. Together with subsegments of
$\beta_1$ and $\beta_3$, it forms a Jordan curve $C\subset U_+\cup\{N\}$ which, near $N$, consists of the two segments.

By Lemma~\ref{lem:star}(ii), $U$ is homeomorphic to an open disc. Thus, by the Jordan curve theorem, $U\setminus C$ has
two components, each of which has $N$ in its closure. A small disc $B$ around $N$ is divided by $C$ into two sectors.
Each of them lies in one component of $U\setminus C$, and each component meets $B\setminus C$; hence, the two sectors lie
in different components. One of them contains points of $\Theta_2$ and the other points of $\Theta_4$. However, $U_-$ is
connected, contains $\Theta_2\cup\Theta_4$, and does not meet $C$. That is a contradiction.
\end{proof}

\subsection{The lower bound}
We now carry out Step~1 of the outline.

\begin{proposition}\label{prop:lower}
For every irreducible representation $\pi$ of $D_b$, we have $\ind_{\pi^+}(S)\ge\dim\pi-[\pi=A_2]$, where $[\pi=A_2]$ is
$1$ if $\pi=A_2$ and $0$ otherwise. Furthermore, $\ind_{A_1^-}(S)\ge1$. In particular, $\Ind(\Sig)\ge2b$.
\end{proposition}

\begin{proof}
Let $R=R_{e_3}$. As $\Rot$ permutes the boundary components, $\Sig$ is not rotationally symmetric about $\ell$; hence, $R\not\equiv0$. By Lemma~\ref{lem:sym}(iii),
\[
  \Rot\cdot R=R,\qquad s\cdot R=-R\ \text{ for the vertical reflections $s$},\qquad \sigma_h\cdot R=R.
\]
That is, $R$ has type $A_2^+$. 

Let $W_1,\dots,W_{2b}$ be the open wedges between the $b$ vertical
mirror planes, and let
\[
  u_k=R\,\mathbf 1_{\Sig\cap W_k}.
\]
Since $s\cdot R=-R$, the function $R$ vanishes on $\Sig\cap P$ for every vertical mirror
plane $P$. Since $|R|$ is $D_b$-invariant, $R\not\equiv0$, and $D_b$ permutes the wedges transitively, $R\not\equiv0$ on
each $\Sig\cap W_k$. Moreover, these sets have full measure in $\Sig$, as $\Sig$ is not contained in a plane. By
Lemma~\ref{lem:pieces}, the functions $u_k$ span a $2b$-dimensional space $V$ with $S|_V=0$ and $V\cap\ker S=\R R$. The group $D_b$ permutes the $2b$ wedges simply
transitively; for $\gamma\in D_b$, let $\gamma(k)$ denote the index with $\gamma(W_k)=W_{\gamma(k)}$. Since
$\mathbf 1_{\Sig\cap W_k}\circ\gamma^{-1}=\mathbf 1_{\Sig\cap W_{\gamma(k)}}$ and $\gamma\cdot R=\det(\gamma)R$, we have
\[
\gamma\cdot u_k=\det(\gamma)\,u_{\gamma(k)} .
\]
Thus, $V$ is the permutation representation of $D_b$ on the wedges, twisted by $A_2$, since $\det\gamma=\chi_{A_2}(\gamma)$ for $\gamma\in D_b$. As the action is
simply transitive, this permutation representation is the regular one (\S\ref{ssec:reps}(c)). Moreover, $\sigma_h$
preserves each wedge and fixes each $u_k$. Therefore, $V\cong(A_2\otimes\cR_b)^+\cong\cR_b^+$, and
Lemma~\ref{lem:iso} gives
\[
  \ind_{\pi^+}(S)\ge m_\pi(\cR_b)-m_\pi(A_2)=\dim\pi-[\pi=A_2].
\]
Summation over all irreducible representations yields
\[
  \sum_\pi\bigl(\dim\pi-[\pi=A_2]\bigr)\dim\pi=2b-1 .
\]
Next, the function $\nu_{e_3}$ has type $A_1^-$ and $\nu_{e_3}\neq 0$ because $\Sig$ is not a flat disc, and
$S(\nu_{e_3},\nu_{e_3})<0$ by Lemma~\ref{lem:basic}(iv). Thus, $\ind_{A_1^-}(S)\ge1$, which gives $\Ind(\Sig)\ge2b$.
\end{proof}

\subsection{Area and energy}\label{ssec:EM}
In this subsection, we establish an equivariant version of the comparison between the second variations of area and
energy \cite{EM08,Lim22}. Here $\Sig$ is any FBMS with induced metric $g$, and $H$ is a finite
group of symmetries of $\Sig$ with $\nu\circ\gamma=\gamma\nu$ for $\gamma\in H$. For an $H$-module $W$, possibly of
infinite dimension, $m_\pi(W)$ denotes the supremum of the multiplicities of $\pi$ in the finite-dimensional submodules
of $W$.

Let $\cV$ be the space of smooth maps $V:\Sig\to\R^3$ with $\langle V,X\rangle=0$ on $\bdSig$, and let
\[
  Q(V,V)=\int_\Sig|\nabla V|^2\dA-\int_{\bdSig}|V|^2\ds .
\]
The group $H$ acts on $\cV$ by $(\gamma\cdot V)(x)=\gamma V(\gamma^{-1}x)$, preserving $Q$. Let $\ind_\pi(Q)$ (respectively, $\ind^0_\pi(Q)$) be the largest
multiplicity of $\pi$ in a finite-dimensional $H$-invariant subspace of $\cV$ on which $Q$ is negative definite ($Q\le0$). 
Furthermore, let
$\cV_1\subset\cV$ be the space of maps $V\in\cV$ with $\Delta V=0$ in $\Sig$ and $\partial_\eta V=V$ on $\bdSig$; that is,
their components are Steklov eigenfunctions with eigenvalue one.

\begin{lemma}\label{lem:energy}
Suppose that $\sigma_1(\Sig)=1$. Then, for every irreducible representation $\pi$ of $H$,
\[
  \ind_\pi(Q)\le m_\pi(\R^3)\qquad\text{and}\qquad\ind^0_\pi(Q)\le m_\pi(\R^3)+m_\pi(\cV_1),
\]
where $H$ acts on $\R^3$ by its defining representation, that is, through the inclusion $H\subset O(3)$.
\end{lemma}

\begin{proof}
Let $W\subset\cV$ be a finite-dimensional $H$-invariant subspace with $Q\le0$ on $W$. The map
\[
  \beta:W\to\R^3,\qquad V\mapsto\int_{\bdSig}V\ds,
\]
is $H$-equivariant. We claim that $\ker\beta\subset\cV_1$. 

Indeed, let $V\in\ker\beta$. Then
each component $V_i$ has boundary mean zero; thus,
\[
  \int_\Sig|\nabla V_i|^2\ge\sigma_1\int_{\bdSig}V_i^2=\int_{\bdSig}V_i^2.
\]
Since $Q(V,V)\le0$, equality holds for each $i$. Hence, $V_i$ is zero or minimizes the Steklov Rayleigh quotient among
functions of boundary mean zero; in both cases, $\Delta V_i=0$ and $\partial_\eta V_i=V_i$. This proves the claim and,
consequently, $m_\pi(W)\le m_\pi(\R^3)+m_\pi(\cV_1)$. Moreover, the argument shows $Q(V,V)=0$ for $V\in\ker\beta$. So if
$Q$ is negative definite on $W$, then $\ker\beta=0$ and $m_\pi(W)\le m_\pi(\R^3)$.
\end{proof}


For a symmetric $2$-tensor $h$ on $\Sig$, let $h^\circ=h-\frac12(\tr_gh)\,g$ be its trace-free part. Let $\mathcal B_\Sig$
be the space of smooth trace-free symmetric $2$-tensors on $\Sig$, and let $\mathcal L_\Sig\subset\mathcal B_\Sig$ be the
subspace of tensors $(L_Yg)^\circ$, where $Y$ is a smooth vector field on $\Sig$ tangent to $\bdSig$ along $\bdSig$ and
$L_Y$ denotes the Lie derivative. We put
\[
  \cT_\Sig=\mathcal B_\Sig/\mathcal L_\Sig ,
\]
and call its elements the \emph{Teichm\"uller directions} of $\Sig$. Intuitively, they are the infinitesimal deformations
of the conformal structure of $\Sig$ modulo diffeomorphisms preserving $\bdSig$. The elements of $H$ are isometries of
$\Sig$ preserving $\bdSig$. Thus, $H$ acts on $\mathcal B_\Sig$ by push-forward, preserving $\mathcal L_\Sig$, and hence on
$\cT_\Sig$. The second fundamental form $A$ lies in $\mathcal B_\Sig$ and satisfies $\gamma_*A=A$ for $\gamma\in H$ by
Lemma~\ref{lem:sym}(iii). Consequently, the map
\[
  \Phi:C^\infty(\Sig)\to\cT_\Sig,\qquad \Phi(f)=[fA],
\]
is $H$-equivariant; that is, $\gamma_*(fA)=(\gamma\cdot f)A$. The following refines an upper bound due to Lima \cite[Thm.~A]{Lim22}; see \cite{EM08} for a closed surface. 
We follow their approach; the differences are that the
tangential field is only required to be tangent to $\bdSig$, which identifies the obstruction space with
$\cT_\Sig$ as an $H$-module, and that the correction is chosen $H$-equivariantly.

\begin{proposition}\label{prop:EM}
For every irreducible representation $\pi$ of $H$,
\[
  \ind_\pi(S)\ \le\ \ind_\pi(Q)+m_\pi(\cT_\Sig)\qquad\text{and}\qquad
  \ind_\pi(S)+\operatorname{null}_\pi(S)\ \le\ \ind^0_\pi(Q)+m_\pi(\cT_\Sig).
\]
\end{proposition}

\begin{proof} Let $f\in C^\infty(\Sig)$, and suppose that $Y$ is a smooth vector field on $\Sig$, tangent to $\bdSig$ along
$\bdSig$, such that $(L_Yg)^\circ=2fA$. Let $V=f\nu+Y$. Then $V\in\cV$, since $\nu$ and $Y$ are orthogonal to $\eta=X$
on $\bdSig$, and, by \cite[Lemma~6.1]{FS16},\[Q(V,V)=S(f,f).\]
Indeed, in their terminology $V$ is a conformal vector field
exactly when $(L_Yg)^\circ=2fA$, since $\langle D_{e_i}V,e_j\rangle=-fA(e_i,e_j)+\langle\nabla_{e_i}Y,e_j\rangle$ and
$\tr A=0$, and $V$ is tangent to $\Sph$ along $\bdSig$. \\

Let $\mathcal N$ be the negative eigenspace of $S$, and let $\mathcal N'$ be either $\mathcal N$ or
$\mathcal N\oplus\ker S$. Then $\mathcal N'$ is finite-dimensional, $H$-invariant, and consists of smooth functions.
Moreover, \[m_\pi(\mathcal N)=\ind_\pi(S) \text{ and } m_\pi(\mathcal N\oplus\ker S)=\ind_\pi(S)+\operatorname{null}_\pi(S).\] Let
$K=\mathcal N'\cap\ker\Phi$. As $\mathcal N'\cong K\oplus\Phi(\mathcal N')$ as $H$-modules,
\[
  m_\pi(\mathcal N')\le m_\pi(K)+m_\pi(\cT_\Sig).
\]
For $f\in K$, there is a smooth vector field $Y$, tangent to $\bdSig$, with $(L_Yg)^\circ=2fA$. Choosing such fields on a
basis of $K$ and extending linearly, we obtain a linear map $f\mapsto Y_f$. We may assume that it is $H$-equivariant;
indeed, we may replace it by $f\mapsto|H|^{-1}\sum_{\gamma\in H}\gamma_*Y_{\gamma^{-1}\cdot f}$, since
$(L_{\gamma_*Y}g)^\circ=\gamma_*(L_Yg)^\circ$ and $\gamma_*(fA)=(\gamma\cdot f)A$. 

Since $\gamma\cdot(f\nu)=(\gamma\cdot f)\nu$,
the map $f\mapsto V_f=f\nu+Y_f$ from $K$ to $\cV$ is linear, injective, and $H$-equivariant. By the calculation above and polarization,
for $f,f'\in K$, \[Q(V_f,V_{f'})=S(f,f').\] Thus, if $\mathcal N'=\mathcal N$, $Q$ is negative definite on the image, which is isomorphic to $K$; $Q\le0$ in the other case. Consequently, $m_\pi(K)\le\ind_\pi(Q)$ (respectively
$\ind^0_\pi(Q)$) and the result follows.
\end{proof}

\subsection{Teichm\"uller directions}\label{ssec:teich}
Here we bound the Teichm\"uller directions of $\Sig\in\cS_b$ of each symmetry type. Towards that end, we
realize $\cT_\Sig$ as a quotient of the space of infinitesimal motions of the caps of a circle domain. This is the most technical part of the paper. 

Let
$\Gamma_1,\dots,\Gamma_b$ be the boundary components of $\Sig$, and let $s_i\in D_b$ be the reflection with
$s_i(\Gamma_i)=\Gamma_i$ (see \S\ref{ssec:reps}); its mirror plane $P_i$ contains $\ell$. The following normal form is
the analogue for $\cS_b$ of the canonical model of \cite[\S4]{ChFS25}.

\begin{lemma}\label{lem:model}
There exist a circle domain $\Om\subset\Sph$ with caps $D_1,\dots,D_b$ and a diffeomorphism $F:\Sig\to\overline\Om$,
conformal for the induced metric, such that $F\circ\gamma=\gamma\circ F$ for all $\gamma\in G$ and $F(\Gamma_i)=\partial D_i$
for all $i$. Moreover, $\gamma(D_i)=D_j$ whenever $\gamma(\Gamma_i)=\Gamma_j$, and the center $c_i$ of $D_i$ lies on the
equator $\ell^\perp\cap\Sph$ and in $P_i$.
\end{lemma}

\begin{proof}
Lemma~\ref{lem:koebe} gives $\Om$, $F$, and $\gamma\mapsto\hat\gamma$ with $F\circ\gamma=\hat\gamma\circ F$. Composing $F$ with
an element $h\in SO(3)$ preserves these properties and replaces $\hat\gamma$ by $h\hat\gamma h^{-1}$. We show that $h$ can be
chosen so that $\hat\gamma=\gamma$ for all $\gamma\in G$.

First, $\hRot$ is a rotation of order $b$, and it fixes $F(N)$, where $N$ is the point of Lemma~\ref{lem:star}(iii).
Differentiating $F\circ \Rot=\hRot\circ F$ at $N$, we see that the conformal linear map $dF_N$ conjugates $d(\Rot)_N$, the
rotation of $T_N\Sig=\ell^\perp$ by the angle $2\pi/b$, to the rotation of $T_{F(N)}\Sph$ by the rotation angle of
$\hRot$. Thus, $\hRot$ is a rotation by $\pm2\pi/b$ about some line. Composing $F$ with a rotation which maps this line onto
$\ell$ and then, if necessary, with the rotation by $\pi$ about a horizontal line, which conjugates $\Rot^{-1}$ to $\Rot$,
we may assume that $\hRot=\Rot$.

Next, the involution $\hat\sigma_h$ has determinant $-1$ and commutes with $\Rot$; thus, it is $-I$ or the reflection in a $\Rot$-invariant plane.
Since $\ell^\perp$ is the only plane invariant under $\Rot$, $\hat\sigma_h$ is $-I$ or $\sigma_h$. It fixes the points of $F(\Lambda)$, where $\Lambda=\Sig\cap\{x_3=0\}\ne\emptyset$
(Lemma~\ref{lem:star}); hence, $\hat\sigma_h=\sigma_h$. 

For a vertical reflection $s\in G$, the involution $\hat s$ has
determinant $-1$ and satisfies
\[
  \hat s\Rot\hat s^{-1}=\Rot^{-1}\ne\Rot.
\]
So it is neither $-I$ nor $\sigma_h$; thus, it is a
reflection in a plane $P'$. It maps the axis $\ell$ of $\Rot$ to the axis of $\Rot^{-1}$; hence, $P'$ contains $\ell$ or is
orthogonal to it. The second case is excluded because $\hat s\ne\sigma_h$. Composing $F$ with a rotation about $\ell$,
which commutes with $\Rot$ and $\sigma_h$ (preserving $\hRot=\Rot$ and $\hat\sigma_h=\sigma_h$), we may assume $\hat s=s$. As $\Rot$, $s$, and $\sigma_h$ generate $G$, it follows
that $\hat\gamma=\gamma$ for all $\gamma\in G$.

We label the caps so that $F(\Gamma_i)=\partial D_i$. Since $\gamma(\overline\Om)=\overline\Om$ and
$\gamma(\partial D_i)=F(\gamma\Gamma_i)$, the element $\gamma$ maps $D_i$ to $D_j$ if $\gamma(\Gamma_i)=\Gamma_j$. In particular,
$\sigma_h$ and $s_i$ preserve $D_i$ and, thus, fix its center $c_i$. The fixed points of $\sigma_h$ and $s_i$ on $\Sph$ form
the equator and $P_i\cap\Sph$, respectively, and the result follows.
\end{proof}

Let $\cC$ be the set of $b$-tuples $\bigl((c_i',r_i')\bigr)_i\in(\Sph\times(0,\pi))^b$ for which the closed discs
$D(c_i',r_i')=\{p\in\Sph:\langle p,c_i'\rangle\ge\cos r_i'\}$ are pairwise disjoint. It is an open subset of a manifold of
dimension $3b$, and $\Om$ defines a point $\omega=\bigl((c_i,r_i)\bigr)_i\in\cC$. The group $G$ acts on $\cC$ by
$(\gamma\cdot\omega')_j=(\gamma c'_i,r'_i)$, where $\gamma(\Gamma_i)=\Gamma_j$. By Lemma~\ref{lem:model}, it fixes $\omega$;
thus, it acts linearly on
\[
  T_\omega\cC=\bigoplus_i\bigl(T_{c_i}\Sph\oplus\R\bigr).
\]

We now define a map from $T_\omega\cC$ to $\cT_\Sig$. Let $g_0$ be the round metric of $\Sph$, and define $\mathcal B_\Om$, $\mathcal L_\Om$, and $\cT_\Om=\mathcal B_\Om/\mathcal L_\Om$
for $(\overline\Om,g_0)$ as in \S\ref{ssec:EM}. Since $F$ is conformal and commutes with $G$, we have $F^*g_0=e^{2w}g$ for
a $G$-invariant function $w$. As $(L_Y(e^{2w}g))^\circ=e^{2w}(L_Yg)^\circ$ for every vector field $Y$, the map
$q\mapsto F_*(e^{2w}q)$ induces a $G$-equivariant isomorphism \[\cT_\Sig\cong\cT_\Om,\] where $G$ acts on $\cT_\Om$ by
push-forward. Thus, it suffices to define a map from $T_\omega\cC$ to $\cT_\Om$.

Towards that end, for $\dot\omega=\bigl((\dot c_i,\dot r_i)\bigr)_i\in T_\omega\cC$, let $\theta_{\dot\omega}:\partial\Om\to\R$ be given by
\[
  \theta_{\dot\omega}(p)=\dot r_i+\frac{\langle\dot c_i,p\rangle}{\sin r_i}\qquad\text{for }p\in\partial D_i .
\]

Let $\mathbf n(p)=(\cos r_i\,p-c_i)/\sin r_i$ be the unit normal of $\partial\Om$ in $\Sph$ pointing into $\Om$. Let
$\omega(t)$ be a curve in $\cC$ and $p(t)\in\partial D(c_i(t),r_i(t))$, both differentiable at $t=0$, with
$\omega(0)=\omega$, $\omega'(0)=\dot\omega$, and $p(0)=p$. Differentiating $\langle p(t),c_i(t)\rangle=\cos r_i(t)$ and
using $\langle p'(0),p\rangle=0$, we obtain
\begin{equation}\label{eq:normal}
  \langle p'(0),\mathbf n(p)\rangle=\theta_{\dot\omega}(p).
\end{equation}
Thus, $\theta_{\dot\omega}$ is the normal speed of the boundary circles, when the caps move with velocity $\dot\omega$.

For $\dot\omega\in T_\omega\cC$, we put
\[
  \Psi(\dot\omega)=\bigl[(L_Zg_0)^\circ\bigr]\in\cT_\Om ,
\]
where $Z$ is any smooth vector field on $\overline\Om$ with $\langle Z,\mathbf n\rangle=\theta_{\dot\omega}$ on $\partial\Om$.
Since two such fields differ by a field tangent to $\partial\Om$, the map \[\Psi:T_\omega\cC\to\cT_\Om\] is well defined and
linear. Intuitively, $\Psi(\dot\omega)$ is the variation of the conformal structure of $\overline\Om$ induced by moving
the caps with velocity $\dot\omega$; one may compare our framework with \cite[\S6]{ChFS25}. 

Let $\mob$ be the space of conformal vector fields on $\Sph$, on which $G$ acts by push-forward. For $\xi\in\mob$, let $\iota(\xi)\in T_\omega\cC$ be the velocity at $t=0$ of the curve
of centers and radii of the discs $\exp(t\xi)(D_i)$; this defines a linear map $\iota:\mob\to T_\omega\cC$.

\begin{lemma}\label{lem:psi}
The following statements hold:
\begin{enumerate}[label=\textup{(\roman*)},itemsep=2pt]
\item the map $\Psi$ is $G$-equivariant;
\item the map $\iota$ is $G$-equivariant and injective, and $\Psi\circ\iota=0$.
\end{enumerate}
\end{lemma}

\begin{proof}
(i) Let $\gamma\in G$, $p\in\partial D_j$, and $\gamma(\Gamma_i)=\Gamma_j$. Then $\gamma^{-1}p\in\partial D_i$ and
$\gamma^{-1}\mathbf n(p)=\mathbf n(\gamma^{-1}p)$. Hence, if $\langle Z,\mathbf n\rangle=\theta_{\dot\omega}$ on $\partial\Om$,
then
\[
  \langle\gamma_*Z,\mathbf n\rangle(p)=\langle Z(\gamma^{-1}p),\mathbf n(\gamma^{-1}p)\rangle=\theta_{\dot\omega}(\gamma^{-1}p)
  =\dot r_i+\frac{\langle\gamma\dot c_i,p\rangle}{\sin r_i}=\theta_{\gamma\cdot\dot\omega}(p).
\]
Thus, $\Psi(\gamma\cdot\dot\omega)=[(L_{\gamma_*Z}g_0)^\circ]=\gamma_*\Psi(\dot\omega)$.

(ii) The flow $\exp(t\xi)$ consists of M\"obius transformations, which map the caps to pairwise disjoint closed round
discs. By \eqref{eq:normal}, applied to $p(t)=\exp(t\xi)(p)$, we have $\theta_{\iota(\xi)}=\langle\xi,\mathbf n\rangle$ on
$\partial\Om$. Hence, $\Psi(\iota(\xi))=[(L_\xi g_0)^\circ]=0$, as $\xi$ is conformal. Equivariance follows from
$\gamma\exp(t\xi)\gamma^{-1}=\exp(t\gamma_*\xi)$ and $\gamma\cdot\omega=\omega$.

It remains to show that $\iota$ is injective. Suppose that $\iota(\xi)=0$. Then $\xi$ is tangent to the $b\ge3$ circles
$\partial D_i$, and its flow preserves each of them. We identify $\Sph$ with $\hat\C=\C\cup\{\infty\}$ by a stereographic
projection. If $\xi\ne0$, then, after conjugation by a M\"obius transformation, its flow is 
\[z\mapsto e^{\kappa t}z \text{ with }
\kappa\in\C\setminus\{0\}, \text{ or } z\mapsto z+t,\] 
since a non-zero element of $\mathfrak{sl}(2,\C)$ is conjugate to a diagonal
or to a nilpotent matrix.

A circle preserved by the flow is a union of orbits and fixed points. For $z\mapsto z+t$, these
circles all contain $\infty$. For $\kappa\notin\mathrm i\R$, they all contain $0$ and $\infty$; they are the lines through
$0$ if $\kappa\in\R$, and there are none otherwise. Thus, for $z\mapsto z+t$ and for $\kappa\notin\mathrm i\R$, no two of
them are disjoint. For $\kappa\in\mathrm i\R$,
they are the circles $|z|=c$, and each closed disc bounded by one of them is of the form $\{|z|\le c\}$ or
$\{|z|\ge c\}\cup\{\infty\}$; thus, no three of these discs are pairwise disjoint. In every case, we reach a contradiction;
hence, $\xi=0$.
\end{proof}

For the surjectivity of $\Psi$, we fix a stereographic projection of $\Sph$ onto $\hat\C$ which maps an interior point of
$D_1$ to $\infty$. Then $\overline\Om\subset\C$,
\[
  D_1=\{|z-a_1|\ge\rho_1\}\cup\{\infty\},\qquad D_i=\{|z-a_i|\le\rho_i\}\quad(i\ge2),
\]
and $g_0=\varrho\,|dz|^2$ with $\varrho>0$. 
Here $a_i$ and $\rho_i$ are the Euclidean centers and radii in $\C$, while the
points of $\cC$ are given by spherical centers and radii as before. Vector fields on $\overline\Om$ are identified with
tangent vector fields of $\Sph$ by the stereographic projection. As usual, we identify a vector field
$Z^1\partial_x+Z^2\partial_y$ on $\overline\Om$ with the complex-valued function $Z=Z^1+\mathrm iZ^2$ of its components;
in complex notation, the vector field is $Z\,\partial_z+\bar Z\,\partial_{\bar z}$. Inner products of tangent vectors are
taken with respect to $g_0$. Since $L_Z\,dz=dZ$, a direct computation gives
\begin{equation}\label{eq:Lie}
  (L_Zg_0)^\circ=2\varrho\operatorname{Re}\bigl(\overline{\partial_{\bar z}Z}\,dz^2\bigr).
\end{equation}
Moreover, every $q\in\mathcal B_\Om$ is of the form $q=2\varrho\operatorname{Re}(\bar\mu\,dz^2)$ for a unique
$\mu\in C^\infty(\overline\Om,\C)$, and $Z$ is tangent to $\partial D_i$ at $z$ if and only if
$\operatorname{Re}\bigl((\bar z-\bar a_i)Z(z)\bigr)=0$.

\begin{lemma}\label{lem:loc}
	Every class in $\cT_\Om$ has a representative with compact support in $\Om$. More precisely, for every
	$q\in\mathcal B_\Om$, there is a smooth vector field $Y$ on $\overline\Om$, tangent to $\partial\Om$, such that
	$q-(L_Yg_0)^\circ$ vanishes on a neighborhood of $\partial\Om$.
\end{lemma}


\begin{proof}
	Write $q=2\varrho\operatorname{Re}(\bar\mu\,dz^2)$. By \eqref{eq:Lie}, a vector field $Y$ satisfies $(L_Yg_0)^\circ=q$ on a
	region exactly when $\partial_{\bar z}Y=\mu$. In particular, holomorphic vector fields do not change the conformal
	structure. Moreover, at the point $a_i+\rho_ie^{\mathrm i\alpha}$ of $\partial D_i$, the radial unit vector is
	$e^{\mathrm i\alpha}$; as the Euclidean inner product of $u,v\in\C$ is $\operatorname{Re}(\bar uv)$, the normal component of
	$Y$ there is $\operatorname{Re}(e^{-\mathrm i\alpha}Y)$, and $Y$ is tangent to $\partial D_i$ exactly when it vanishes. 
	
	Thus,
	we have to solve $\partial_{\bar z}Y=\mu$ near $\partial\Om$ subject to $\operatorname{Re}(e^{-\mathrm i\alpha}Y)=0$ on each
	circle. The strategy is to solve the equation and then correct the boundary values by holomorphic fields near each circle.
	
	\emph{Step 1:} Extend $\mu$ to a smooth function with compact support in $\C$, and let
	\[
	Z_0(z)=\frac1\pi\int_\C\frac{\mu(w)}{z-w}\,dA(w)
	\]
	be its Cauchy transform. Since $1/(\pi z)$ is a fundamental solution of $\partial_{\bar z}$, the function $Z_0$ is smooth
	and satisfies $\partial_{\bar z}Z_0=\mu$ on $\C$ \cite[Thm.~1.2.2]{Hor90}; see also \cite[Ch.~V]{Ahl06}. In general, $Z_0$
	is not tangent to $\partial\Om$. Its normal component on $\partial D_i$,
	\[
	\psi_i(\alpha)=\operatorname{Re}\bigl(e^{-\mathrm i\alpha}Z_0(a_i+\rho_ie^{\mathrm i\alpha})\bigr)
	=\sum_{n\in\mathbb Z}\hat\psi_{i,n}e^{\mathrm in\alpha},
	\]
	is a smooth real function; in particular, $\hat\psi_{i,-n}=\overline{\hat\psi_{i,n}}$.
	
	\emph{Step 2:} For each $i$, we look for a function $h_i$, holomorphic on the side of
	$\partial D_i$ on which $\Om$ lies, with the same normal component as $Z_0$, that is, with
	$\operatorname{Re}(e^{-\mathrm i\alpha}h_i)=\psi_i$ on $\partial D_i$. We write $h_i=\rho_i^{-1}(z-a_i)\,k_i$; as the factor
	$\rho_i^{-1}(z-a_i)$ equals $e^{\mathrm i\alpha}$ on $\partial D_i$, the condition becomes $\operatorname{Re}k_i=\psi_i$ on
	$\partial D_i$, a Dirichlet problem for the real part of a holomorphic function. 
	
	As $\infty\in D_1$, the domain $\Om$ lies
	inside $\partial D_1$ and outside $\partial D_i$ for $i\ge2$, and we set
	\[
	k_1=\hat\psi_{1,0}+2\sum_{n\ge1}\hat\psi_{1,n}\Bigl(\frac{z-a_1}{\rho_1}\Bigr)^{n},\qquad
	k_i=\hat\psi_{i,0}+2\sum_{n\ge1}\hat\psi_{i,-n}\Bigl(\frac{\rho_i}{z-a_i}\Bigr)^{n}\quad(i\ge2).
	\]
	At $z=a_i+\rho_ie^{\mathrm i\alpha}$, the powers in these series are $e^{\mathrm in\alpha}$ and $e^{-\mathrm in\alpha}$,
	respectively; so, using $\hat\psi_{i,-n}=\overline{\hat\psi_{i,n}}$, we obtain $\operatorname{Re}k_i=\psi_i$. Since $\psi_i$
	is smooth, its Fourier coefficients decay faster than any power of $n$. Hence, the series converge on $\overline\Om$
	together with all their derivatives, and $h_i$ is smooth on $\overline\Om$ and holomorphic in $\Om$. Consequently,
	$Z_0-h_i$ is tangent to $\partial D_i$ and satisfies $\partial_{\bar z}(Z_0-h_i)=\mu$.
	
	\emph{Step 3:} Choose $\chi_i\in C^\infty(\overline\Om)$ equal to $1$ near $\partial D_i$ and to $0$ near
	$\partial D_j$ for $j\ne i$, and let $Y=Z_0-\sum_i\chi_ih_i$. Near $\partial D_i$, we have $Y=Z_0-h_i$; so $Y$ is tangent to
	$\partial\Om$, and $\partial_{\bar z}Y=\mu$ near $\partial\Om$. Elsewhere,
	$\partial_{\bar z}Y=\mu-\sum_i(\partial_{\bar z}\chi_i)h_i$ differs from $\mu$ only on the support of the differentials
	$d\chi_i$, which is a compact subset of $\Om$. By \eqref{eq:Lie},
	\[
	q-(L_Yg_0)^\circ=2\varrho\operatorname{Re}\bigl(\overline{(\mu-\partial_{\bar z}Y)}\,dz^2\bigr)
	\]
	vanishes near $\partial\Om$. As $Y$ is tangent to $\partial\Om$, we have $(L_Yg_0)^\circ\in\mathcal L_\Om$, so this tensor
	represents the class $[q]$.
\end{proof}


\begin{theorem}\label{thm:surj}
	Let $\Om\subset\Sph$ be a circle domain with $b\ge1$ caps, and let $\omega\in\cC$ and $\Psi:T_\omega\cC\to\cT_\Om$ be defined as
	above. Then $\Psi$ is surjective. 
\end{theorem}


\begin{proof}
	By Lemma~\ref{lem:loc}, it suffices to show that $[q]\in\Psi(T_\omega\cC)$ whenever $q=2\varrho\operatorname{Re}(\bar\mu\,dz^2)$
	with $\mu$ smooth and compactly supported in $\Om$. By the definition of $\Psi$ and by \eqref{eq:Lie}, it suffices to find
	$\dot\omega\in T_\omega\cC$ and a smooth vector field $Z$ on $\overline\Om$ such that
	\begin{equation}\label{eq:surjgoal}
		\text{(a)}\ \ \partial_{\bar z}Z=\mu\ \text{ in }\Om,\qquad\text{(b)}\ \ \langle Z,\mathbf n\rangle=\theta_{\dot\omega}
		\ \text{ on }\partial\Om .
	\end{equation}
	
	The idea is as follows. For small $t$, the conformal structure on $\Om$ with
	Beltrami coefficient $t\mu$ is, by Koebe's theorem, conformally equivalent to a circle domain $\Om_t$. If the uniformizing
	maps $\phi_t$ depended differentiably on $t$, with $\phi_0=\mathrm{id}$, then $Z=\partial_t\phi_t|_{t=0}$ would satisfy (a),
	by differentiating the Beltrami equation, and (b), since $\phi_t$ maps the circles of $\Om$ to those of $\Om_t$. Koebe's
	theorem does not provide this differentiability, so we construct the maps $\phi_t$ explicitly, using circle inversions and the measurable Riemann mapping theorem.
	
	\emph{Step 1:} Let $\mathcal I_i$ be the inversion in the circle $\partial D_i$; that is, the
	anticonformal involution of $\hat\C$ fixing $\partial D_i$ pointwise, and let $\cK$ be the group generated by
	$\mathcal I_1,\dots,\mathcal I_b$. We claim that the sets $\gamma(\Om)$, $\gamma\in\cK$, are pairwise disjoint.
	
	Indeed,
	$\mathcal I_i$ maps $\hat\C\setminus D_i$, which contains $\Om$ and the caps $D_j$ with $j\ne i$, into the interior of $D_i$.
	Hence, for every word $\gamma=\mathcal I_{i_1}\cdots\mathcal I_{i_k}$ with $k\ge1$ and $i_m\ne i_{m+1}$, applying the
	letters from right to left, we see that $\gamma(\Om)\subset\operatorname{int}D_{i_1}$. Every element of $\cK$ is represented
	by such a word or by the empty word; so $\gamma(\Om)\cap\Om=\emptyset$ for $\gamma\ne1$. For $\gamma, \gamma'\in \cK$, applying the argument above to
	$\gamma'^{-1}\gamma$ proves the claim. Moreover, $\mathcal I_i(\Om)=\operatorname{int}D_i\setminus\bigcup_{j\ne i}\mathcal I_i(D_j)$ and each $I_i(D_j)$ is a compact subset of $\operatorname{int}D_i$. Hence, a neighborhood of $\partial D_i$ is
	contained in $\Om\cup\partial D_i\cup\mathcal I_i(\Om)$ and meets no other copy.
	
	
	\emph{Step 2:} For real $t$ with $|t|\sup|\mu|<1$, we define a conformal structure $\mathfrak c_t$ on $\hat\C$ in three pieces. On $\Om$,
	it is the conformal structure of the metric $|dz+t\mu\,d\bar z|^2$, whose Beltrami coefficient is $t\mu$; the condition on
	$t$ ensures that this metric is positive definite. On each copy $\gamma(\Om)$ with $\gamma\in\cK\setminus\{1\}$, it is the
	image of this structure under $\gamma$; that is, $\gamma$ is conformal from $(\Om,\mathfrak c_t)$ onto
	$(\gamma(\Om),\mathfrak c_t)$. On the rest of $\hat\C$, it is the standard structure. 
	
	By Step~1, $\mathfrak c_t$ is well
	defined; it is only measurable, since it has infinitely many pieces accumulating at the limit set. It is $\cK$-invariant, since each $\mathcal I_i$ permutes the
	copies. The Beltrami
	coefficient of the image of a conformal structure under a conformal or anticonformal map depends real-linearly on the
	original one and has the same modulus. Hence, the Beltrami coefficient of $\mathfrak c_t$ is $t\tilde\mu$ for a bounded
	measurable function $\tilde\mu$ independent of $t$, with $\tilde\mu=\mu$ on $\Om$ and $|\tilde\mu|\le\sup|\mu|$. Since $\mu$
	\emph{vanishes} near $\partial\Om$ and, by Step~1, only $\Om$ and $\mathcal I_i(\Om)$ meet a neighborhood of $\partial D_i$, the
	function $\tilde\mu$ vanishes on a neighborhood of $\partial\Om$, and it is smooth on a neighborhood $\Om'$ of $\overline\Om$.
	
	\emph{Step 3:} By the measurable Riemann mapping theorem and its dependence on parameters
	\cite{AB60} (see also \cite[Ch.~V]{Ahl06}), there are quasiconformal homeomorphisms $\phi_t$ of $\hat\C$ fixing $0$, $1$,
	and $\infty$, with Beltrami coefficient $t\tilde\mu$; that is,
	\[
	  \partial_{\bar z}\phi_t=t\tilde\mu\,\partial_z\phi_t,
	\]
	so that
	$\phi_t$ maps $\mathfrak c_t$ to the standard conformal structure. Moreover, $\phi_0=\mathrm{id}$, $t\mapsto\phi_t(z)$ is
	differentiable at $t=0$ for every $z$, and $\dot\phi=\partial_t\phi_t|_{t=0}$ is continuous with $\partial_{\bar z}\dot\phi=\tilde\mu$
	in the sense of distributions.
	
	\emph{Step 4:} Let $J=\phi_t\mathcal I_i\phi_t^{-1}$. As $\mathcal I_i$ preserves $\mathfrak c_t$ and
	$\phi_t$ maps $\mathfrak c_t$ to the standard structure, $J$ is an anticonformal homeomorphism of $\hat\C$. Its composition
	with complex conjugation is a $1$-quasiconformal homeomorphism of $\hat\C$ and, hence, a M\"obius transformation
	\cite{Ahl06}; so $J$ is an anti-M\"obius transformation. It fixes the Jordan curve $\phi_t(\partial D_i)$ pointwise. After
	conjugating three of its fixed points to $0$, $1$, and $\infty$, it becomes $z\mapsto\bar z$, since a M\"obius
	transformation fixing $0$, $1$, and $\infty$ is the identity. So the fixed points of $J$ form a circle, which must contain,
	and hence be equal to, the Jordan curve $\phi_t(\partial D_i)$. 
	
	Consequently, $\phi_t$ maps the caps $D_i$ onto closed
	round discs and $\Om$ onto a circle domain; this is an explicit uniformization of $(\Om,\mathfrak c_t)$ by a circle domain.
	The discs $\phi_t(D_i)$ define a point $\omega(t)\in\cC$ with $\omega(0)=\omega$, and $\omega(t)$ is differentiable at
	$t=0$, because the center and the radius of a circle depend smoothly on three of its points. Let $\dot\omega=\omega'(0)$.
	
	\emph{Step 5:} For $p\in\partial\Om$, the curve $p(t)=\phi_t(p)$ lies on the boundary circles of
	$\omega(t)$; so \eqref{eq:normal} gives $\langle\dot\phi,\mathbf n\rangle=\theta_{\dot\omega}$ on $\partial\Om$. Next, let
	$\chi\in C^\infty_c(\Om')$ be equal to $1$ near $\overline\Om$. On $\Om'$, the function $\dot\phi$ differs from the Cauchy
	transform of $\chi\tilde\mu$ by a continuous function $h$ with $\partial_{\bar z}h=0$ near $\overline\Om$ in the sense
	of distributions. Such a function is holomorphic, since its mollifications are holomorphic and converge to it locally
	uniformly. As the Cauchy transform of $\chi\tilde\mu$ is smooth, so is
	$\dot\phi$ on $\overline\Om$, and $\partial_{\bar z}\dot\phi=\tilde\mu=\mu$ in $\Om$. Thus, $Z=\dot\phi$ satisfies
	\eqref{eq:surjgoal}, and, by \eqref{eq:Lie},
	\[
	\Psi(\dot\omega)=[(L_{\dot\phi}g_0)^\circ]=[2\varrho\operatorname{Re}(\bar\mu\,dz^2)]=[q].
	\]
	
	In summary, $\Psi(\dot\omega)$ is the derivative of the uniformization of the deformed conformal structures by circle
	domains.
\end{proof}
\begin{remark}
	The proof follows a classical pattern: the quasiconformal
	deformation technique for groups generated by reflections, as in the deformation theory of Kleinian groups
	\cite{AB60}.
\end{remark}

We are now ready to bound the Teichm\"uller directions of each type.

\begin{proposition}\label{prop:teich}
For every irreducible representation $\pi$ of $D_b$,
\[
  m_{\pi^+}(\cT_\Sig)\le\dim\pi-[\pi=A_2]-[\pi=E_1]\qquad\text{and}\qquad m_{A_1^-}(\cT_\Sig)=m_{E_1^-}(\cT_\Sig)=0 .
\]
\end{proposition}

\begin{proof}
By Lemma~\ref{lem:psi} and Theorem~\ref{thm:surj}, $\cT_\Sig\cong\cT_\Om\cong T_\omega\cC/\ker\Psi$ as $G$-modules, and $\ker\Psi$
contains a copy of $\mob$. Hence, for every irreducible representation $\pi$ of $G$,
\[
  m_\pi(\cT_\Sig)\le m_\pi(T_\omega\cC)-m_\pi(\mob).
\]

We first compute $T_\omega\cC$. By Lemma~\ref{lem:model}, $T_{c_i}\Sph=\R e_3\oplus\R v_i$, where $v_i$ is a unit normal of
$P_i$. The group $G$ permutes the summands $T_{c_i}\Sph\oplus\R$ of $T_\omega\cC$ as it permutes the boundary components,
and $\gamma$ maps $(\dot c,\dot r)$ at $c_i$ to $(\gamma\dot c,\dot r)$. Thus, $\sigma_h$ acts by $-1$ on the vectors $e_3$
and by $+1$ on the vectors $v_i$ and on the radii. 

The $\sigma_h$-odd part $\bigoplus_i\R e_3$ is permuted by $D_b$ like
the boundary components; hence, it is isomorphic to $\cP_b$. The trace of $\gamma\in D_b$ on the $\sigma_h$-even part
$\bigoplus_i(\R v_i\oplus\R)$ is the sum of its traces on the summands with $\gamma(\Gamma_i)=\Gamma_i$. It is $2b$ for
$\gamma=1$, and it is $0$ for a non-trivial rotation, which fixes no boundary component. It is also $0$ for a reflection
$s$: every summand with $s(\Gamma_i)=\Gamma_i$ has $s=s_i$, and $s_i$ acts by $-1$ on $v_i$ and by $+1$ on the radius.
Therefore, the $\sigma_h$-even part has the character of $\cR_b$, and
\[
  T_\omega\cC\cong\cR_b^+\oplus\cP_b^- .
\]

Next, the conformal vector fields of $\Sph$ are the fields $\xi_{a,c}(p)=a\times p+c-\langle c,p\rangle p$ with
$a,c\in\R^3$, and $\gamma_*\xi_{a,c}=\xi_{\det(\gamma)\gamma a,\gamma c}$ for $\gamma\in O(3)$. So
$\mob\cong(\R^3\otimes\det)\oplus\R^3$. Since $\R^3\cong E_1^+\oplus A_1^-$, $\det$ restricts to $A_2$ on $D_b$ and to $-1$
on $\sigma_h$, and $E_1\otimes A_2\cong E_1$, we obtain
\[
  \mob\cong E_1^+\oplus A_1^-\oplus E_1^-\oplus A_2^+ .
\]

Combining the above, we conclude that
\[
  m_{\pi^+}(\cT_\Sig)\le m_\pi(\cR_b)-[\pi=E_1]-[\pi=A_2]=\dim\pi-[\pi=A_2]-[\pi=E_1],
\]
and, by \eqref{eq:Pb}, $m_{\pi^-}(\cT_\Sig)\le m_\pi(\cP_b)-[\pi=A_1]-[\pi=E_1]=0$ for $\pi\in\{A_1,E_1\}$.
\end{proof}

\begin{remark}\label{rem:caps}
By Lemma~\ref{lem:model}, $\Sig$ is conformally the sphere with $b$ equal round caps removed whose centers form a
regular $b$-gon on the equator (Figure~\ref{fig:caps}); such a domain is determined by the common radius. The proof of
Proposition~\ref{prop:teich} counts the motions of the caps modulo the M\"obius motions. Changing the radius or the
longitude of a cap is $\sigma_h$-even, while changing its latitude is $\sigma_h$-odd. The M\"obius motions are the
rotations about $\ell$ (type $A_2^+$), the rotations about horizontal axes, which tilt the equator ($E_1^-$), the vertical
boost, which changes the latitude of all caps together ($A_1^-$), and the horizontal boosts, which shrink the caps on one
side and enlarge those on the other ($E_1^+$). This leaves at most $(2b-3)+(b-3)=3b-6$ directions, the dimension of the
Teichm\"uller space of a sphere with $b$ holes.
\end{remark}

\subsection{The odd sector}\label{ssec:odd}
In this subsection, we analyze the $\sigma_h$-odd functions by restricting them to the upper half. Let $U$ and $\Lambda$
be as in Lemma~\ref{lem:star}, let $H^1_\Lambda(U)$ be the space of functions $u\in H^1(U)$ vanishing on $\Lambda$, and let
\[
  S_U(u,u)=\int_U\bigl(|\nabla u|^2-|A|^2u^2\bigr)\dA-\int_{\bdSig\cap\overline U}u^2\ds .
\]
Restriction to $U$ is an isomorphism from the space of $\sigma_h$-odd functions in $H^1(\Sig)$ onto $H^1_\Lambda(U)$, whose
inverse is the odd extension. Moreover, $S(f,f)=2S_U(f|_U,f|_U)$ and $\int_\Sig f^2=2\int_Uf^2$ for odd $f$. Let
$\lambda_1\le\lambda_2\le\cdots$ be the eigenvalues of $S_U$ on $H^1_\Lambda(U)$ relative to $L^2(U)$. The group $D_b$
preserves $U$, $\Lambda$, and $S_U$.

\begin{lemma}\label{lem:SU}
The following statements hold:
\begin{enumerate}[label=\textup{(\roman*)},itemsep=2pt]
\item the odd extension of an eigenfunction of $S_U$ with eigenvalue $\lambda$ is an eigenfunction of $S$ with
      eigenvalue $\lambda$, and every $\sigma_h$-odd eigenfunction of $S$ arises in this way; in particular, the
      eigenfunctions of $S_U$ are smooth on $\overline U$, and the odd extension maps a $D_b$-invariant space of
      eigenfunctions of $S_U$ of type $\pi$ to one of type $\pi^-$;
\item the eigenvalue $\lambda_1$ is simple, with a $D_b$-invariant eigenfunction $\varphi_1>0$ in $U$;
\item every eigenfunction $f$ of $\lambda_2$ has exactly two nodal domains, that is, components of
      $\{x\in\overline U:f(x)\ne0\}$, and satisfies $f(N)\ne0$ or $df_N\ne0$.
\end{enumerate}
\end{lemma}

\begin{proof}
(i) Let $u$ be an eigenfunction of $S_U$ with eigenvalue $\lambda$, and let $f$ be its odd extension. Then
$S(f,h)=\lambda\int_\Sig fh$ for all odd $h$, and trivially for all even $h$, since both sides vanish by symmetry. Thus,
$f$ is an eigenfunction of $S$ and, consequently, smooth up to $\bdSig$. The converse and the statement on types are
clear.

(ii) Let $u$ be an eigenfunction of $\lambda_1$. Then $u$ minimizes $S_U(v,v)/\int_Uv^2$ on
$H^1_\Lambda(U)\setminus\{0\}$, and so does $|u|$; thus, $|u|$ is also an eigenfunction of $\lambda_1$. By the strong
maximum principle, $|u|>0$ in $U$; as $U$ is connected, $u$ has a sign in $U$. Since two functions with a sign in $U$
cannot be $L^2$-orthogonal, $\lambda_1$ is simple, with an eigenfunction $\varphi_1>0$. For $\gamma\in D_b$, the function
$\gamma\cdot\varphi_1$ is again a positive eigenfunction of $\lambda_1$ with the same $L^2$-norm; hence,
$\gamma\cdot\varphi_1=\varphi_1$.

(iii) The nodal domains of $f$ are relatively open in $\overline U$ and disjoint from $\Lambda$, and $f$ vanishes on their
boundaries relative to $\overline U$. As $f$ is orthogonal to $\varphi_1$, it changes sign; thus, there are at least two
nodal domains.

We claim that there are at most two. Suppose, to the contrary, that there are at least three, and let
$\mathcal O_1,\mathcal O_2,\mathcal O_3$ be three of them. Let $f_k=f\mathbf 1_{\mathcal O_k}$ (cf.\ \cite[Ch.~VI, \S6]{CH53}).
The truncation argument in the proof of Lemma~\ref{lem:pieces} shows $f_k\in H^1_\Lambda(U)$. For $k\ne k'$, we have
$f_kf_{k'}=0$ on $\overline U$ and $\langle\nabla f_k,\nabla f_{k'}\rangle=0$ almost everywhere; hence,
$S_U(f_k,f_{k'})=0$. Similarly, $S_U(f-f_k,f_k)=0$, and thus
\[
  S_U(f_k,f_k)=S_U(f,f_k)=\lambda_2\int_Uf_k^2\dA .
\]
Choose $(c_1,c_2)\ne0$ such that $w=c_1f_1+c_2f_2$ is $L^2$-orthogonal to $\varphi_1$. Then
$S_U(w,w)=\lambda_2\int_Uw^2$. Consequently, $w$ minimizes $S_U(v,v)/\int_Uv^2$ among the functions
$v\in H^1_\Lambda(U)\setminus\{0\}$ orthogonal to $\varphi_1$ and, therefore, is an eigenfunction of $\lambda_2$. However,
$w\ne0$ vanishes on the open set $\mathcal O_3\cap U$, which contradicts unique continuation \cite{Aro57}. This proves the
claim.

Therefore, there are exactly two nodal domains, one on which $f>0$ and one on which $f<0$. Their intersections with $U$
are connected. Indeed, a connected relatively open subset of the closed disc $\overline U$ remains connected when its
points on $\partial\overline U$ are removed, since these points have arbitrarily small neighborhoods whose intersections
with $U$ are connected. As $\Delta f+(|A|^2+\lambda_2)f=0$, Lemma~\ref{lem:jordan} gives $f(N)\ne0$ or $df_N\ne0$.
\end{proof}

Using the representation theory at $N$, we can now control the second odd eigenvalue.

\begin{proposition}\label{prop:odd}
Let $\Sig\in\cS_b$, and suppose that $\ind_{A_1^-}(S)\le1$ and $\ind_{E_1^-}(S)=0$. Then $\lambda_1<0\le\lambda_2$.
Consequently, the $\sigma_h$-odd part of the negative space of $S$ is spanned by the odd extension of $\varphi_1$.
\end{proposition}

\begin{proof}
By Lemma~\ref{lem:SU}(i), the $\sigma_h$-odd part of the negative space of $S$ corresponds to the span of the
eigenfunctions of $S_U$ with negative eigenvalues. By Proposition~\ref{prop:lower}, it is non-zero; thus, $\lambda_1<0$.

We prove $\lambda_2\ge0$ by contradiction. Suppose that $\lambda_2<0$, and let $\mathcal E$ be its eigenspace. It is
$D_b$-invariant and orthogonal to $\varphi_1$. Moreover, $\mathcal E$ contains no subrepresentation of type $A_1$;
otherwise, together with the odd extension of $\varphi_1$, the negative space of $S$ would contain two copies of $A_1^-$.
It contains none of type $E_1$ either, since $\ind_{E_1^-}(S)=0$. On the other hand, the maps $f\mapsto f(N)\in\R$ and
$f\mapsto df_N\in T^*_N\Sig$ are $D_b$-equivariant, where $D_b$ acts trivially on $\R$ and, by Lemma~\ref{lem:star}(iii),
through $E_1$ on $T^*_N\Sig$. By Schur's lemma, both maps vanish on $\mathcal E$, which contradicts
Lemma~\ref{lem:SU}(iii). Hence, $\lambda_2\ge0$.
\end{proof}

We are now ready to compute the index.

\begin{proposition}\label{prop:index}
Let $\Sig\in\cS_b$. Then the negative space of $S$ is isomorphic to $(\cR_b\ominus A_2)^+\oplus A_1^-$ as a
representation of $G$. In particular, $\Ind(\Sig)=2b$.
\end{proposition}

\begin{proof}
By \cite[Thm.~1.1]{Seo24} (see also \cite[Thm.~1.1]{CG26}), $\sigma_1(\Sig)=1$. By Proposition~\ref{prop:EM} with $H=G$ and Lemma~\ref{lem:energy}, for
every irreducible representation $\pi$ of $G$,
\[
  \ind_\pi(S)\le m_\pi(\R^3)+m_\pi(\cT_\Sig).
\]
Since $\R^3\cong E_1^+\oplus A_1^-$, we have $m_{\pi^+}(\R^3)=[\pi=E_1]$, $m_{A_1^-}(\R^3)=1$, and $m_{E_1^-}(\R^3)=0$.
Together with Proposition~\ref{prop:teich}, it gives
\[
  \ind_{\pi^+}(S)\le\dim\pi-[\pi=A_2]\quad\text{for every $\pi$},\qquad \ind_{A_1^-}(S)\le1,\qquad \ind_{E_1^-}(S)=0 .
\]
Combining with Proposition~\ref{prop:lower}, we obtain $\ind_{\pi^+}(S)=\dim\pi-[\pi=A_2]$ for every $\pi$ and
$\ind_{A_1^-}(S)=1$. Furthermore, by Proposition~\ref{prop:odd}, the $\sigma_h$-odd part of the negative space is
one-dimensional and, hence, of type $A_1^-$. Therefore, the negative space of $S$ is isomorphic to
$(\cR_b\ominus A_2)^+\oplus A_1^-$, and $\Ind(\Sig)=(2b-1)+1=2b$.
\end{proof}

\subsection{The nullity}\label{ssec:null}
Finally, we determine the kernel of $S$. We first identify the Steklov eigenfunctions with eigenvalue one; the following
lemma is a special case of results of Karpukhin, Kusner, McGrath, and Stern \cite{KKMS24}.

\begin{lemma}\label{lem:steklov}
	Let $\Sig\in\cS_b$. Then the Steklov eigenfunctions of $\Sig$ with eigenvalue one (the solutions of $\Delta u=0$ in
	$\Sig$ with $\partial_\eta u=u$ on $\bdSig$) are the linear combinations of $X_1$, $X_2$, and $X_3$.
\end{lemma}

\begin{proof}
By Lemma~\ref{lem:star}(ii), $(\Sig,\sigma_h)$ is a basic reflection surface in the sense of \cite[Def.~5.4]{KKMS24},
since $\Sig/\sigma_h$ is a disc whose boundary is the image of $\bdSig\cup\Lambda$; so the first Steklov eigenspace of $\Sig$ has dimension at most three
\cite[Props.~5.12, 5.15]{KKMS24}. Since $\sigma_1(\Sig)=1$ \cite[Thm.~1.1]{Seo24} (see also \cite{CG26}), it is the space in question, and it
contains $X_1,X_2,X_3$, which are linearly independent as $\Sig$ is not planar.
\end{proof}

\begin{proposition}\label{prop:null}
Let $\Sig\in\cS_b$. Then $\ker S$ is spanned by $R_{e_1}$, $R_{e_2}$, and $R_{e_3}$. In particular, $\Null(\Sig)=3$.
\end{proposition}

\begin{proof} We divide the argument into several steps. 
	
\emph{Step 1: the space $\cV_1$.} By Lemma~\ref{lem:steklov}, every $V\in\cV_1$ is of the form $V=MX$ for a real
$3\times3$ matrix $M$. The condition $\langle V,X\rangle=0$ on $\bdSig$ says that the symmetric part $M^s$ of $M$ satisfies
$\langle M^sx,x\rangle=0$ for $x\in\bdSig$. The zero set on $\Sph$ of a non-zero quadratic form is empty, a pair of
antipodal points, a great circle, the union of two great circles, or, for an indefinite form of rank three, the union of
two disjoint closed curves. In each case, it contains at most two pairwise disjoint Jordan curves. As $\bdSig$ consists
of $b\ge3$ disjoint Jordan curves, $M^s=0$. Hence, $\cV_1=\{a\times X:a\in\R^3\}$. Since $a\mapsto a\times X$ is injective
and $\gamma\cdot(a\times X)=\det(\gamma)\,\gamma a\times X$, we obtain
\[
  \cV_1\cong\R^3\otimes\det\cong E_1^-\oplus A_2^+\qquad\text{as $G$-modules.}
\]

\emph{Step 2: the Jacobi fields $R_a$.} The linear map $a\mapsto R_a$ is injective. Indeed, its kernel is a
$\Rot$-invariant subspace consisting of vectors which span axes of rotational symmetry of $\Sig$. It does not contain $e_3$,
as $\Rot$ permutes the boundary components. If it contained a vector $a\notin\ell$, it would contain the independent vectors
$a$ and $\Rot a$; then $\Sig$ would be invariant under the rotations about two different axes and, hence, under $SO(3)$,
which is absurd. By Lemma~\ref{lem:basic}(iii) and Lemma~\ref{lem:sym}(iii), $R_{e_3}\in\ker S$ has type $A_2^+$, and
$R_{e_1},R_{e_2}\in\ker S$ span a copy of $E_1^-$.

\emph{Step 3: the upper bound.} By Proposition~\ref{prop:EM} with $H=G$, Lemma~\ref{lem:energy}, and
Proposition~\ref{prop:teich}, for every irreducible representation $\pi$ of $G$,
\[
  \ind_\pi(S)+\operatorname{null}_\pi(S)\le m_\pi(\R^3)+m_\pi(\cV_1)+m_\pi(\cT_\Sig).
\]
For $\pi^+$, the right-hand side is at most
\[
  [\pi=E_1]+[\pi=A_2]+\dim\pi-[\pi=A_2]-[\pi=E_1]=\dim\pi .
\]
Since $\ind_{\pi^+}(S)=\dim\pi-[\pi=A_2]$ by Proposition~\ref{prop:index}, it follows that
$\operatorname{null}_{\pi^+}(S)\le[\pi=A_2]$. For $A_1^-$, the right-hand side is $1=\ind_{A_1^-}(S)$; thus,
$\operatorname{null}_{A_1^-}(S)=0$. For $E_1^-$, it is $1$; thus, $\operatorname{null}_{E_1^-}(S)\le1$.

It remains to exclude the other $\sigma_h$-odd types. By Lemma~\ref{lem:SU}(i), the $\sigma_h$-odd part of $\ker S$
corresponds to the eigenspace of $S_U$ for the eigenvalue $0$. By Propositions~\ref{prop:index} and \ref{prop:odd},
$\lambda_1<0\le\lambda_2$. Moreover, the restriction of the $\sigma_h$-odd Jacobi field $R_{e_1}$, which is non-zero by Step~2, has
eigenvalue $0$.
Hence, $\lambda_2=0$, and, by Lemma~\ref{lem:SU}(iii), the map $f\mapsto(f(N),df_N)$ is injective on the $\sigma_h$-odd
part of $\ker S$. Since this map is $D_b$-equivariant with values in a representation isomorphic to $A_1\oplus E_1$, only
the types $A_1^-$ and $E_1^-$ can occur in the $\sigma_h$-odd part of $\ker S$.

Consequently, $\dim\ker S\le3$. Together with Step~2, we conclude that
\[
  \ker S=\Span\{R_{e_1},R_{e_2},R_{e_3}\}.\qedhere
\]
\end{proof}

\begin{proof}[Proof of Theorem~\ref{thm:main}]
It follows from Propositions~\ref{prop:index} and \ref{prop:null}.
\end{proof}

\end{document}